\documentclass[preprint,11pt]{elsarticle}
\usepackage{natbib, changes}
\usepackage{url}
\usepackage{bm}
\usepackage{amsthm,amsmath,amsfonts,amssymb}
\usepackage{tikz}
\usepackage{graphicx}
\usepackage{enumitem}
\usepackage{verbatim}
\fontsize{10pt}{10pt}\selectfont
\fontsize{9pt}{9pt}\selectfont

\newtheorem{Theorem}{Theorem}
\newtheorem{Lemma}[Theorem]{Lemma}
\newtheorem{Corollary}[Theorem]{Corollary}
\newtheorem{Proposition}[Theorem]{Proposition}

\theoremstyle{remark}
\newtheorem{Example}{Example}
\newtheorem{Remark}[Example]{Remark}

\def\pgf{\emph{p.g.f.~}}

\def\iid{\emph{i.i.d.}~}

\def\a{{\mathfrak a}}
\def\A{{\mathfrak A}}
\def\B{{\mathfrak B}}
\def\b{{\mathfrak b}}

\def\w{{\mathfrak w}}
\def\m{{\mathfrak m}}
\def\v{{\mathfrak v}}
\def\u{{\mathfrak u}}
\def\p{{\mathfrak p}}

\def\n{{\mathfrak n}}
\def\v{{\mathfrak v}}
\def\M{{\mathfrak M}}
\def\V{{\mathfrak V}}
\def\K{{\mathfrak{K}}}
\usepackage[para]{footmisc} 
\providecommand{\keywords}[1]
{
  \small	
  \textbf{\textit{Keywords---}} #1
}

\def\pgf{\emph{p.g.f.~}}

\def\iid{\emph{i.i.d.}~}

\newcommand*{\dt}{%
{\mbox{\raisebox{0.07em}{\large\bfseries .}}}}

\newpageafter{abstract}
\begin{document}
\begin{frontmatter}
\title{
A Class of Long Range  Circulant Random Walks on $\mathbb{Z}_q^d$\\
}
 \author{Robert Griffiths}
\affiliation{
organization={School of Mathematics},
            addressline={Monash University},
            country={Australia}
         }
	
\ead{Bob.Griffiths@Monash.edu}
 \author{Shuhei Mano}
\affiliation{
organization={The Institute of Statistical Mathematics},
            addressline={Tokyo},
            country={Japan}
         }	
\ead{smano@ism.ac.jp}

\date{}

\begin{abstract}
This paper studies a class of long range random walks 
$(X_t)_{t=0}^\infty$ on the direct product of cyclic groups $\mathbb{Z}_q^d$ for $d\ge 1$ and $q\ge 2$. 
$X_{t+1} = X_t + Z_t \mod q$,
with $(Z_t)_{t=1}^\infty$ \iid on $\{0,1,\ldots, q-1\}^d$.
Entries of $Z_t$ are updated by circulant matrices, possibly with dependence.
Multiple entries of $Z_t$ can be non-zero in a transition.
An emphasis is on finding the structure of such random walks and spectral expansions for the transition functions.
An extension is made to processes on a $d$-dimensional torus,  scaling entries in $\{0,1,\ldots, q-1\}$ by dividing by $q$ and letting $q\to \infty$.
The state space is then $d$ circles of unit perimeter, where $0$ and $1$ are identified as the same point in each circle.
\color{black}

 If the entries of $X_t$ are exchangeable then a grouping of $X_t$ is made by taking counts of the types $0,\ldots, q-1$ in $X_t$. In this grouping the multivariate Krawtchouk polynomials become the eigenvectors. Examples consider cutoff times and mixing times in these processes. A limit form for the multivariate Krawtchouk polynomials is used to find a central limit theorem for the transition distributions in the grouped model as $d \to \infty$.
\medskip

\noindent
\keywords \ Circulant random walks, Random walk on a Hamming graph, Long range random walk, multivariate Krawtchouk polynomials.
\end{abstract}
%
\end{frontmatter}

\section{Introduction}
A random walk on 
the direct product of additive cyclic groups, $\mathbb{Z}_q^d$ for $d\ge 1$ and $q\ge 2$, \color{black}
is constructed as
\begin{equation}
X_{t+1} = X_t + Z_t \mod q,
\label{rweqn:000}
\end{equation}
where $(Z_t)_{t=1}^\infty$ are \iid on $\{0,1,\ldots, q-1\}^d$. 
Entries of $Z_t$ are updated by circulant transition probability matrices.
Denote ${\cal V}_q = \{0,1,\ldots,q-1\}$ and ${\cal V}_{q,d}={\cal V}_q^d$.
${\cal V}_{q,d}$ is the set of elements of $\mathbb{Z}_q^d$. 
These random walks are interesting because they contain and generalise simple random walks derived from Ehrenfest urn models with one or more urns, and simple random walks on Hamming graphs.
In this generalisation to long range random walks simultaneous updates of many coordinates can be made where multiple entries of $Z_t$ are non-zero.
The entries of $Z_t$ are not necessarily independent or identically distributed.
The eigenstructure of the transition functions of $(X_t)$ naturally leads to spectral analysis on $\mathbb{Z}_q^d$. If the entries of $Z_t$ have an exchangeable distribution then the eigenfunctions in this class of random walks are multivariate Krawtchouk polynomials which are also interesting in representation theory.
If states  $j \in {\cal V}_q$ are identified with complex numbers $\theta_j=e^{2\pi i/q\dt  j}$, $i=\sqrt{-1}$, then adding an increment $Z_t$ is equivalent to multiplying the corresponding complex states.
Denote a complex transform acting on $j \in {\cal V}_q$ as
$
\widehat{j} = \theta_j=e^{2\pi i/q\dt j}
$.
Then (\ref{rweqn:000}) is equivalent to a Hadamard product
\begin{equation}
\widehat{X}_{t+1} = \widehat{X}_t \odot \widehat{Z}_t.
\label{rweqn:00}
\end{equation}
If $q=2$ the states of $\widehat{X}_t$ are $e^{2\pi i\ 0/2}=1$
and $e^{2\pi i\ 1/2}= e^{\pi i\ } = -1$ corresponding to spins. Adding increments which are $0$ or $1$ is the same as multiplying by spins to get another spin in a transition. In a model on ${\cal V}_q$ the $q^{\text{th}}$ roots of unity $e^{2\pi i/q\dt  j}$ extend the role of $-1,1$ when $q=2$. The set $\{e^{2\pi i/q\dt  j}\}_{j=0}^{q-1}$ is a multiplicative cyclic group $C_q$ with generator $g=e^{2\pi i/q}$.
\color{black}
The stationary distribution in the random walks (\ref{rweqn:000}) is uniform on  ${\cal V}_{q,d}$, however it is possible that
the random walk is not ergodic
because of the initial configuration and/or dependence structure of $Z_t$. There could be periodic behaviour or the full state space of ${\cal V}_{q,d}$ may not be generated.
Long range random walks when $d=2$ and $(X_t)$ is a random walk in $\{0,1\}^\mathbb{Z}$ are studied in \cite{CG2021}. Mixing times to the stationary distribution of $d$-dimensional independent  Bernoulli$(1/2)$ random variables are found.

If $d=1$ the random walk (\ref{rweqn:000}) is 1-dimensional with state space ${\cal V}_q$ and has a circulant transition matrix. Then conditional on $X_t=a$, $X_{t+1}=a + V \mod q$, where $V$ is a random variable in ${\cal V}_q$ which has the distribution of $Z_t$.
Circulant transition matrices are well studied and their eigen-structure, which is important, is well known.

An emphasis in this paper is on finding and understanding spectral expansions of random walks (\ref{rweqn:000}) on $\mathbb{Z}_q^d$.

In Section \ref{Circulant} we give a characterization of \emph{all} possible transition functions of Markov chains which have eigenfunctions $(\theta_r^x)$, the same as in a circulant random walk. This is  shown to be equivalent to a characterization of possible eigenvalues which have the form $\mathbb{E}\big [\theta_r^V]$ where $V$ is a non-negative random variable on ${\cal V}_q$.

In Subsection \ref{ctsstate} the process on ${\cal V}_q$ is extended to a process on a circle of unit perimeter, where $0$ and $1$ are identified as the same point in each circle. Entries in ${\cal V}_q$ are scaled by dividing by $q$ and letting $q\to \infty$,  where $0$ and $1$ are identified as the same point. This is a particularly interesting process.

In Section \ref{MCVN} the full long range random walk (\ref{rweqn:000}) on $\mathbb{Z}_q^d$ is studied. The eigenvectors of these processes have the form $(\theta_1^{x\dt r})$ where $x$ is in the state space of the Markov chain and $r$ is an index, $x,r \in {\cal V}_{q,d}$ and $x\dt r$ is the inner product of $x$ and $r$. They constitute the set of characters of irreducible representations of $\mathbb{Z}_q^d$. 
An extension of the characterization in 1-dimension shows that \emph{all} such Markov chains  have eigenvalues
$\big (\mathbb{E}\big [\theta_1^{V\dt r}\big ]\big )$, where $V \in {\cal V}_{q,d}$ is a non-negative random variable. In Subsection \ref{qexpansion:0} the entries of $(X_t)$ are assumed to be exchangeable. Then we let
$M_t = (M_{t}[j])_{j=0}^{q-1}$, where $M_{t}[j]$ counts the number of entries equal to $j$ in $X_t$.
The transition function for $M_t=m$ to $M_{t+1}=n$ has a spectral expansion with eigenvalues which are multivariate Krawtchouk orthogonal polynomials 
$\big (Q_l(n;(\theta_k)\big )$
on the stationary 
Multinomial$(d,1/q)$ distribution. These polynomials have been studied by various authors. They were first constructed by \cite{G1971},  then later articles about them are \cite{DG2014,DG2019}.
The particular form that we use is one where the elementary basis is $(\theta_1^r)$, which 
have been studied by \cite{M2004}, and appeared in \cite{M2018} as eigenfunctions arising from stochastic urn models. \citeauthor{M2004} takes an algebraic approach of harmonic analysis which is briefly discussed in this paper. 
The multivariate Krawtchouk polynomials are particularly interesting and we spend effort in presenting them in a self-contained way.
In Remark \ref{Hamming:23} a random walk in $ {\cal V}_{q,d}$ is considered where transitions between $x,y \in {\cal V}_{q,d}$ only depend on a Hamming distance defined by $\partial(x,y) = |\{k:x[k]\ne y[k], k \in [d]\}|$.
This model is important in a quantum walk in \cite{GM2025}.
It is a random walk on $H(d,q)$, a Hamming graph, studied by Hora \cite{H1997}. The treatment in \cite{GM2025}, \cite{H1997} \color{black} is algebraic, whereas our treatment is probabilistic. 

 A limit for the polynomials when $d \to \infty$ is found in Lemma \ref{Kplimit}, which leads to calculations of cutoff time in one model in Example \ref{simple:000} and mixing times bounds in another Example \ref{bigchi}.
The existence of a cutoff time in the model of Example \ref{simple:000} is consistent with a condition for cutoff to exist in 
Cayley random walks on groups in \cite{HOT2026}.
A characterization of transition functions which have these multivariate Krawtchouk polynomials as eigenfunctions is found in
Proposition \ref{Generaltrans}. In Subsection \ref{CLTQ}, Lemma \ref{QCLT} a central limit theorem form of the multivariate Krawtchouk polynomials is derived. 

$(X_t)$ can be infinite dimensional provided $(Z_t)$ is well defined.  A brief account is made when the entries of $Z_t$ are from a de Finetti sequence.

The principal contribution of this paper is a probabilistic study of the structure of a class of long range random walks on $\mathbb{Z}^d_q$ which have multivariate Krawtchouk polynomials as eigenfunctions. This approach complements the existing algebraic studies of \cite{M2004,MT2004,M2011,M2018} that we connect with. Random walks that we study have transition functions which belong to a convex set. The eigenvalues of transition functions in this convex set are characterized. Applications are made to cutoff times in the random walks and a central limit theorem when $d\to \infty$.
\section{Circulant random walks}\label{Circulant}
Let $P$ be a $q\times  q$ circulant transition probability matrix, which is real,  
with first non-negative row $v=(v_j)_{j=0}^{q-1}$.  A circulant matrix $P$ has the form
\begin{equation}
P=
\begin{bmatrix}
v_0&v_1&\cdots &v_{q-1}\\
v_{q-1}&v_0&\cdots&v_{q-2}\\
\cdots&\cdots&\cdots &\cdots\\
\cdots&\cdots&\cdots &\cdots\\
v_1&v_2&\cdots &v_0
\end{bmatrix}.
\label{circulantmatrix:00}
\end{equation}
The rows correspond to clockwise rotations in ${\cal V}_q$.  It is straightforward to see that the stationary distribution of $P$ is uniform on ${\cal V}_q$ because $eP=e$, where $e$ is a row vector of unit entries.
 Denote $\theta_l = e^{2\pi i/q\dt  l}$, as the $l^{\text{th}}$ power of the $q^{\text{th}}$ root of unity.
The eigenvector pairs and eigenvalues of a circulant transition matrix (\ref{circulantmatrix:00})
are respectively $(\theta_r^a,\theta_r^{-b})$ and
\begin{equation}
\eta_r= \sum_{l=0}^{q-1}v_l\theta_r^l.
\label{eigenv:00a}
\end{equation}
A spectral expansion of $P$ is therefore
\begin{align}
P_{ab} = \frac{1}{q}\sum_{r=0}^{q-1}\eta_r \theta_r^a\theta_r^{-b}
= \frac{1}{q}\big \{1 + \sum_{r=1}^{q-1}\eta_r \theta_r^a\theta_r^{-b}\big \}.
\label{spectral:00}
\end{align}
 See, for example,  \cite{F1991} Vol 1, XVI.2.
The eigenvectors are $(\theta_r)_{r=0}^{q-1}$, orthogonal on the uniform distribution for 
$a \in {\cal V}_q$. Complex conjugates are taken in the orthogonal sum.
\[
\frac{1}{q}\sum_{a=0}^{q-1}\theta_r^a\bar{\theta}_s^a=
\frac{1}{q}\sum_{a=0}^{q-1}
e^{2\pi i/q  \dt ar}e^{-2\pi i/q\dt  as}
= \begin{cases}
1&\text{if~}r=s\\
\frac{1}{q}\cdot
\frac{ e^{2\pi i\ (r-s)}-1 }{e^{2\pi i/q\dt (r-s)}-1 } = 0
&\text{if~}r\ne s
\end{cases}.
\]
We could write the eigenvalues (\ref{eigenv:00a}) as expected values taking $V$ to be a random variable 
such that $\mathbb{P}\big (V=j \big ) = v_j$, 
$j\in {\cal V}_q$ so
\begin{equation}
\eta_r = \mathbb{E}\big [\theta_r^V\big ].
\label{eigenv:00aa}
\end{equation}
Then
\begin{equation}
P_{ab} = \mathbb{E}\Big [\frac{1}{q}
\sum_{r=0}^{q-1}\theta_r^V\theta_r^a\theta_r^{-b}\Big ] 
= \mathbb{E}\big [\frac{1}{q}
\sum_{r=0}^{q-1}\theta_1^{r(V+a-b)}\big ].
\label{spectral:01}
\end{equation}
$Z_t$ in (\ref{rweqn:00}) has the distribution of $V$.
A notation used later for entries of a random transition matrix is
\begin{equation}
P_{ab}(V) = \frac{1}{q}\sum_{r=0}^{q-1}\theta_r^V\theta_r^a\theta_r^{-b}.
\label{PV:00}
\end{equation}
Conditional on $V=v$ 
\begin{equation}
P_{ab}(v) = 
\begin{cases}
1&\text{if~}b=a+v \mod q\\
0&\text{otherwise}
\end{cases},
\label{PVP:00}
\end{equation}
though not immediately obvious from (\ref{PV:00}).

Let $R$ be uniform on ${\cal V}_q$. A way to write (\ref{spectral:01}) as real valued, knowing $P_{ab}$ is real is
\begin{equation*}
P_{ab}=\mathbb{E}\big [\cos(2\pi/q\ \dt R(V+a-b)\big ]
\end{equation*}
where expectation is over $R$ and $V$. 
\subsection{Characterization of transition functions}
A Lancaster characterization of transition functions with the same eigenvectors as 
(\ref{spectral:01}), showing that \emph{all} such transition matrices are mixtures of circulant matrices follows. Characterizing transition functions with given eigenfunctions is equivalent to characterizing possible eigenvalues of the transition matrix.
\begin{Proposition}\label{pcirc:0}
$P^\circ$ with entries
\begin{equation*}
P^\circ_{ab} = \frac{1}{q}
\sum_{r=0}^{q-1}\rho_r\theta_r^a\theta_r^{-b}
\end{equation*}
is non-negative and therefore a transition probability matrix if and only if
$ \rho_r =  \eta_r = \mathbb{E}\big [\theta_r^V\big ]$ for a random variable $V$ on ${\cal V}_q$.\\
The probability distribution of $V$ satisfies
\begin{equation}
\mathbb{P}\big (V=v\big )
= \frac{1}{q}\sum_{r=0}^{q-1}\eta_r\theta_r^{-v}.
\label{inv:000}
\end{equation}
\end{Proposition}
\begin{proof}
The sufficiency is clear from the circulant construction with $\rho_r = \eta_r$.\\ For the necessity let $B$ given $A=a$ be random variables distributed as $P^\circ_{ab}$, $b\in{\cal V}_q$.
\[
\mathbb{E}\big [\theta_r^B\mid a\big ] = \rho_r\theta_r^a.
\]
Setting $a=0$
\[
\rho_r = \mathbb{E}\big [\theta_r^B\mid a=0\big ]
\]
so $V$ in $\eta_r$ is taken to have the distribution of $B\mid a=0$.\\
The series expansion (\ref{inv:000}) is an inverse Fourier transform of 
(\ref{eigenv:00aa}) so (\ref{inv:000}) holds.
\end{proof}
The set of transition probability matrices ${\cal P}= \big \{P^\circ{(\rho_r)}\big \}$ is a convex set because if $P^\circ{(\rho_r)},P^\circ{(\rho^\prime_r)} \in {\cal P}$ then
$\lambda P^\circ{(\rho_r)} + (1-\lambda)P^\circ{(\rho^\prime_r)} \in {\cal P}$ for $\lambda \in [0,1]$.
Extreme points of ${\cal P}$ are when $\rho_r = \theta_r^v$ and then $V=v$ with probability 1.
In these extreme points $P_{ab} = 1$ if $b=a+V\!\!\! \mod q$ and zero otherwise.
They correspond to Markov chains where 
\begin{equation*}
\widehat{X}_{t+1}= \widehat{X}_t\theta_v.
\end{equation*}
If $v=0$ there is no change in transitions because $\theta_0=1$.
Another eigenvalue property is that if $P^\circ{(\rho_r)},P^\circ{(\rho^\prime_r)} \in {\cal P}$, then $P^\circ{(\rho_r\rho_r^\prime)} \in {\cal P}$.
This follows because 
\[
P^\circ_{ab}(\rho_r\rho_r^\prime )= \sum_{c=0}^{q-1}
P^\circ_{ac}(\rho_r)P^\circ_{cb}(\rho_r^\prime).
\]
\begin{Proposition} A stationary distribution for the random walk (\ref{rweqn:00}) is the uniform distribution on ${\cal V}_q$.
\end{Proposition}
\begin{proof}
The stationary distribution is clear because if $P$ is a circulant matrix (\ref{circulantmatrix:00}) and $e$ a row vector of unit entries, then $eP=e$.
\end{proof}
Here are some examples of 1-dimensional random walks which are well known.
\begin{Example} A random walk when $q=2$.\\ $V$ is Bernoulli$(v_1)$, and  the $2^{\text{th}}$ roots of unity are $\eta_0=\theta_0=1, \eta_1=\theta_1=-1$. $X_{t+1} = X_t + V_t \mod 2$, where $(V_t)$ are \iid.
  Then $\widehat{X}_{t+1} = \widehat{X}_t\cdot (-1)^{V_t}$. $X_t$ has state space $\{0,1\}$ and $\widehat{X}_t$ has state space $\{-1,1\}$.

\end{Example}
\begin{Example}\label{Example:1}
Jumps are one step to the right. $v_1=1$. The 'random walk' is deterministic. Transitions are made
\[
0\to 1,\ 1 \to 2,\ \ldots ,\ q-2\to q-1,\ q-1 \to 0.
\]
If $X_t=l$ then $X_{t+1}=l+1$.
, equivalently $\widehat{X}_t=\theta_l$ so $\widehat{X}_{t+1} = \theta_l\theta_1$ and 
$\widehat{X}_t = \theta_1^t$ taking $\widehat{X}_0=1$.
The eigenvalues are $\eta_r = \theta_r.$
\end{Example}
\begin{Example}\label{Example:2}
This example is of a symmetric random walk with lazyness and unit jumps. Let $\gamma \in [0,1]$.
Jumps are one unit right or left with equal probability. Then
$v_0=1-\gamma,\ v_1= \gamma/2,\  v_{q-1}= \gamma/2$ and so
\[
\eta_r = (1-\gamma)\theta_1^0 
+ \frac{\gamma}{2}(\theta_1^1 + \theta_1^{-1})
= 1-\gamma + \gamma \cos (2\pi r/q).
\]
It is convenient to take subscripts of $v$\!\!\! $\mod q$ so $v_{-1} \equiv v_{q -1 }$.
The random walk is ergodic if $\gamma < 1$ or $\gamma=1$ and $q$ is odd.\\
\end{Example}
\begin{Example}\label{Example:3}
Uniform on ${\cal V}_q$.
\[
\eta_r = \frac{1}{q}\sum_{k=0}^{q-1}\theta_r^k = \delta_{r0}.
\]
The Markov chain then mixes in one step because 
$P^\circ_{ab}= \frac{1}{q}$ for all $b\in {\cal V}_q$.
This example is a special case of Example 2.2 in \cite{GM2025}.
\end{Example}
\begin{Example}\label{Example23}
Jumps are made one unit to the right or left with probabilities
$
\beta\gamma/2,\beta\gamma/2
$ or  with probability $1-\beta\gamma$ a uniform choice is made on ${\cal V}_q$.\\
Then 
\[
\eta_r =
\begin{cases}
 1&r=0\\
 \beta\gamma \cos (2\pi r/q)
 &r=1,\ldots q-1
  \end{cases}.
 \]
 The eigenvalues $\eta_r$ are a mixture of those in Examples \ref{Example:2} and
 \ref{Example:3}.\\
 This is an ergodic random walk.
\end{Example}
\begin{Example}\label{Example:symmetric}
Symmetric positive and negative jumps.
If $v_j=v_{-j}$ then 
\begin{equation}
\eta_r = v_0 + \sum_{l=1}^{q-1}v_le^{2\pi i/q \dt   lr}\
= v_0 + \sum_{l=1}^{q-1}v_{q-l}e^{2\pi i/q \dt  lr}\
= v_0 + \sum_{l=1}^{q-1}v_le^{-2\pi  i/q\dt lr}.
\label{sum:a00}
\end{equation}
Averaging the first and last expressions in (\ref{sum:a00})
\begin{equation}
\eta_r = v_0 + \sum_{l=1}^{d-1}v_l\cos(2\pi lr/q).
\label{symjump:00}
\end{equation}
Note that $\eta_0=1$. The eigenvalues are then real and $\eta_r=\eta_{q-r}$. When the eigenvalues are real $P$ is Hermite.
There are $\lfloor \frac{q-1}{2}\rfloor$ repeated eigenvalues
  and if $q$ is even there is a single eigenvalue $\eta_{q/2}$. 
$P$ has a real spectral expansion with 
\begin{align*}
P_{ab} &= \frac{1}{q}\big \{1 + \sum_{r=1}^{q-1}\eta_r\cos(2\pi r(a-b)/q)\big \}\label{symjump:01a}
\\
&= \frac{1}{q}\big \{1 + \sum_{r=1}^{q-1}\eta_r 
\big (\cos (2\pi r a/q)\cos (2\pi r b/q) + 
\sin(2\pi ra/q)\sin(2\pi rb/q) \big )\big \}.
\end{align*}
The expression  follows by taking the real part of $P_{ab}$ in (\ref{spectral:00}) knowing that the eigenvalues are real.
Note the reversibility condition that $P_{ab}=P_{ba}$ $\iff$ $V$ has symmetric positive and negative jumps $\iff$ $(\eta_r)$ is real. If $q$ is prime, the random walk is ergodic (Proposition \ref{Ergodic}). If $v_1=\cdots = v_{q-1}$
 it reduces to Example 2.1 in \cite{GM2025} with $d=1$. It is an ergodic random walk.
\end{Example}
\begin{Proposition}\label{Ergodic}
    A sufficient condition for a reversible random walk to be ergodic is that $q$ is prime and larger than 2, or $q=2$ with $v[0]>0$.
\end{Proposition}

\begin{proof} Exclude the non-trivial case when $v[0]=1,\ v[j]=0,\ j=1,\ldots d-1$.\\
 \emph{Aperiodicity.} Since the eigenvalues are real and we
  have \eqref{symjump:00}, the random walk is periodic if
  and only if
  $\rho_r=\mathbb{E}[\cos(2\pi Vr/q)]=-1$, almost surely
  for some $r\in\{1,\ldots,q-1\}$. This condition is
  equivalent to $2Vr/q$ being an odd integer almost surely.
  Suppose $q=2$. Then, $r=1$ and $V=1$, almost surely. Otherwise,
  $Vr/q$ never becomes an integer if $V>0$ because $q$
  is a prime. 
  Therefore, the random walk is
  aperiodic if $q$ is prime and larger than 2, or
  $q=2$ with $v_0>0$.\\
  \emph{Irreducibility.} Let the support of $V$ be $S(V)$.
  Suppose $S(V)\subsetneq\mathcal{V}_{q}$. The random
  walk is irreducible otherwise. Then, the random walk
  is reducible if and only if all of the elements of
  $S(V)$ divide $q$. If $q$ is a prime, no element
  of $V$ divides $q$. Therefore, if $q$ is a prime
  the random walk is irreducible.\\
  \emph{Ergodicity.} Because of the above arguments on
  aperiodicity and irreducibility, the assertion
  follows.
\end{proof}  

%

\label{ctstorussection}
 \subsection{A process on the torus with continuous state space}\label{ctsstate}
 To obtain a continuous state space limit on the torus let
$\a = a/q, \b = b/q, \V = V^{(q)}/q$ and keep them fixed as $q\to \infty$. Then we have a weak limit to a process $(X_t)_{t\in \mathbb{N}}$
on a torus (circle in one dimension). An analogy with the discrete circulant matrix process is that instead of a distribution $V$ on ${\cal V}_q$ there is a continuous random variable $\V$ on $[0,1]$. 
We consider the torus group
 $\mathbb{T}=\mathbb{R}/\mathbb{Z}$, which
 is identified with the unit circle
 $S^1=\{|z|\in\mathbb{C}:|z|=1\}$ by the bijection
 $\mathbb{T}\to S^1: \v \mod 1\mapsto e^{2\pi i \v}$.
 A transition from $\a\in [0,1]$ is made to $\b=\a+\V, \mod 1$  taken as a position in $\mathbb{T}$.
Denote the process by $(\B)_{t\in \mathbb{N}}$.
An eigenvector/eigenvalue equation is
\begin{equation}
\mathbb{E}\big [e^{2\pi i r \B_{t+1}}\mid \B_t=\a\big ]
 = \mathbb{E}\big [e^{2\pi i r(\a+\V)}\big ]
= e^{2\pi i r\a}\cdot \mathbb{E}\big [e^{2\pi i r\V}\big ]
.
\label{ctseig:00} 
\end{equation}

An analogue of Proposition \ref{pcirc:0} is the following.
\begin{Proposition}\label{pcirc:1}
Let $(B_t)_{t\in \mathbb{N}}$ be a homogeneous stochastic process on $\mathbb{T}$ with eigenvectors $(e^{2\pi i r B_t})_{r\in \mathbb{N}}$ and a uniform stationary distribution on $[0,1]$. Then the class of such processes is characterized by having eigenvalues
\begin{equation}
\rho_r = \mathbb{E}\big [e^{2r\pi i\V}],
\label{propcts:000}
\end{equation}
where $\V$ is a random variable on $[0,1]$, and for $t\in \mathbb{Z}_+$, $r \in \mathbb{N}$,
\begin{equation*}
\mathbb{E}\big [e^{2\pi i r \B_{t}}\mid \B_{t-1}=\a\big ]
= \rho_re^{2\pi i r\a}
.
\end{equation*}

\end{Proposition}
\begin{proof}
The sufficiency is clear from (\ref{ctseig:00}).
For the necessity suppose that
\begin{equation}
\mathbb{E}\big [e^{2\pi i r \B_{1}}\mid \B_0=\a\big ]
= \rho_re^{2\pi i r\a}
.
\label{ctseig:000} 
\end{equation}
Set $\a=0$ in (\ref{ctseig:000}) to obtain
\[
\rho_r = \mathbb{E}\big [e^{2\pi i r \B_{1}}\mid \B_0=0\big ].
\]
Then $\V$ is identified as having the distribution of $B_1\mid B_0=0$.
\end{proof}
\begin{Example}\label{Example:8}
If $\V$ is uniform on $[0,1]$ then $(\B_t)$, with a continuous state space, mixes in one step  because $\rho_r = \delta_{r,0}$, similar to in
Example \ref{Example:3}.
\end{Example}

 
There is a question as to whether $({\cal B}_t)$ is recurrent or transient.
An example shows that periodic behaviour can occur.
\begin{Example}
Suppose that $\V=1/3$ with probability 1. Let $\a\in [0,1]$. The state space is continuous in this example.
Since $\V$ is constant 
${\cal B}_t$ is the remainder in $\a+t/3$. If $t$ is a multiple of $3$ then ${\cal B}_t=\a$
and the process is recurrent and periodic. Suppose that $\V$ is a fixed irrational number in $[0,1]$ with probability 1.
Then $({\cal B}_t)$ is not recurrent.
\end{Example}
%

We consider the case when an eigenfunction expansion of the transition density exists. Note that a probability density $f$ on $\mathbb{T}$ is demanded so that the density is defined on $\v \mod 1$ $\in\mathbb{T}$, or equivalently, $f$ is a periodic function:
\[
  f(\v)=f(\v+n), \quad \forall n\in\mathbb{Z}, ~ \forall \v \in [0,1].
\]
Let $f_t\big (\b\mid \a\big )$ be the transition density of $\B_{t+1}\mid B_t=\a$.
Taking the limit $q\to\infty$ in \eqref{spectral:01}:
\begin{equation}
P_{ab}
  =\mathbb{E}\Big[\frac{1}{q}\sum_{r=0}^{q-1}\theta_1^{r(V+a-b)}\Big]
  =\mathbb{E}\Big[\frac{1}{q}\sum_{r=-\lfloor q/2\rfloor}^{q-1-\lfloor q/2\rfloor}\theta_1^{r(V+a-b)}\Big],
\label{seq:cts}  
\end{equation}
a density expansion is
\begin{align}
f_t\big (\b\mid \a\big )
&=  \sum_{r={-\infty}}^{\infty}
\mathbb{E}\big [e^{2\pi i\ \V r}\big ]^t
e^{2\pi i\ \a r}e^{-2\pi i\ \b r},\ \a,\b \in \mathbb{T},
\label{cts:20}
\end{align}
which is real. 
If $q$ is odd (\ref{cts:20}) follows immediately by \eqref{seq:cts}.
If $q$ is even, \eqref{seq:cts} is
\[
  \frac{1}{q}\theta_1^{-q(V+a-b)/2}+
  \mathbb{E}\Big[\frac{1}{q}\sum_{r=-q/2+1}^{q/2-1}
    \theta_1^{r(V+a-b)}\Big]
\]
and the first term vanishes in the limit.
\color{black}
In this example we can show that the series expansion (\ref{cts:20}) converges for $t\ge 1$ under an assumption on the smoothness of the density.
For simplicity, set $t=1$ in (\ref{cts:20}). Let
\[
\hat{g}(r)=\mathbb{E}\Big[e^{2\pi i\V r}\Big ]
=\int_0^1 f_1(\v|0)e^{2\pi ir\v}d\v
=\mathbb{E}\Big[e^{2\pi ir\B_1}|\B_0=0\Big].
\]
and $\V$ be  continuous with density $g(\v)$, $\v \mod 1 \in\mathbb{T}$, which has bounded second derivatives. Because of periodicity $g(0)=g(1)$ and $g^\prime(0)= g^\prime(1)$. Integrating by parts twice
\[
  \int_0^1 g^{\prime\prime}(\v)e^{2\pi i\v r} d\v = 
  -(2\pi r)^2 \int_0^1 g(\v)e^{2\pi i\v r} d\v=
  -(2\pi r)^2 \hat{g}(r).
\]
Then
\[
  \Big|\int_0^1 g^{\prime\prime}(\v)e^{2\pi i\v r} d\v\Big| 
  \le \int_0^1 |g^{\prime\prime}(\v)|d\v\le C, \quad
  C=\sup_{\v\in\mathbb{T}} |g''(\v)|,  
\]
giving a bound
\[
  |\hat{g}(r)|\le \frac{C}{(2\pi r)^2} \quad \text{for} \quad r\neq 0.
\]
Therefore the series converges in $\ell^1$, because it is bounded by
\[
  1+\frac{2C}{(2\pi)^2}\sum_{r=1}^\infty\frac{1}{r^2}<\infty.
\]
Since the transform of the series (\ref{cts:20}) is proper when $t=1$, the series expansion of  $f_1(\b\mid \a)$ is well defined.
We can see that the uniform distribution is the stationary distribution,
since
\begin{align*}
  \int_0^1 1\cdot f_1(\b|\a)d\a&=
  \sum_{r=-\infty}^\infty\mathbb{E}\Big[ e^{2\pi i\V r}\Big]
  e^{-2\pi i \b r}\int_0^1 e^{2\pi i\a r}d\a\\
  &=\sum_{r=-\infty}^\infty\mathbb{E}\Big[e^{2\pi i\V r}\Big]
  e^{-2\pi i \b r}\delta_{r,0}=1.
\end{align*}

\begin{Example}
An example of a circular distribution is the von-Mises distribution with the density
\[
  g(\v)=\frac{e^{k\cos(2\pi\v)}}{I_0(k)}, \quad  \v\!\!\! \mod 1 \in \mathbb{T},
\]
and then
\[
\hat{g}(r)=\mathbb{E}[e^{2\pi i\V r}]=\frac{I_r(k)}{I_0(k)}, \quad
r\in\mathbb{Z},
\]
where $I_r(k)=I_{-r}(k)$ is the modified Bessel function of the first kind of order $r$, and the argument $k$ controls concentration of the density at $0$. If $k=0$ the density reduces to the uniform density. We know $\hat{g}(0)=1$ and $\hat{g}(r)$ are positive and monotonically decreasing with the increase of $|r|$. Therefore, we can see the convergence $f_t(\b|\a)\to 1$ as $t\to\infty$. The stationary distribution is the uniform distribution.
\end{Example}

\begin{Remark}
  Even if the transition density does not exist, the formal series expansion in \eqref{cts:20} has a Physics interpretation. If $\V=v$, a.s., we have
  \[
  \sum_{r=-\infty}^{\infty}
  e^{2\pi ir(vt+\a-\b)},\ \a,\b \in {\mathbb{T}.}
  \]
  This series does not converge. The interpretation is that 
  there is a closed vibrating string,
  with waves propagating in the positive direction. For each mode $r$, $2\pi r$ {are} the wave numbers, which are $2\pi$ times the inverse of wave length, {and $rv$ is the frequency.}
\end{Remark}
\medskip

There is an extension of Proposition \ref{pcirc:1}  characterizing continuous time processes where eigenvalues have the form (\ref{cts:08}).
The proof in continuous time is essentially due to \cite{B1949} in a context for processes with different eigenfunctions than ours. 
Bochner's proof is also used in \cite{G2009} for characterisations of eigenvalues in a class of stochastic processes.  
\begin{Proposition}
Let $(X_\tau)_{\tau \in \mathbb{R}_+}$ on $\mathbb{T}$  be a Poisson embedding to continuous time with $X_\tau = \B_{N_\tau}$ where $N_\tau$ is a Poisson process of rate $\lambda$.  An eigenvector/eigenvalue equation for $(X_\tau)$ is
\begin{align*}
\mathbb{E}\big [e^{2\pi i r X_\tau}\mid X_0=x_0\big ]
&= 
\exp \big \{
\lambda\tau \mathbb{E}\big [e^{2\pi ir \V}-1\big ] \}
\cdot e^{2\pi i r x_0},
\end{align*}
so the eigenvalues are $\exp \big \{
\lambda\tau \mathbb{E}\big [e^{2\pi ir \V}-1\big ]\big \}$.\\ There is a more general form of the eigenvalues.\\
If the probability measure of
$\V$ is $\varphi_\lambda$ then suppose that $\phi$ is the weak limit
$\lambda \varphi_\lambda(d\v) \to \phi(d\v)/\v$, $\v >0$.
 Let $(Z_\tau)_{\tau \geq 0}$ be a subordinator on $[0,1]$ with characteristic function
\begin{equation*}
\mathbb{E}\big [e^{iZ_\tau\vartheta }\big ] = 
\exp \Big \{\tau \int_{[0,1]}\frac{e^{i\v\vartheta } -1}{\v}\phi(d\v)
\Big \}.
\end{equation*}
Eigenvalues in the subordinated process $(X_{Z_\tau})_{\tau \geq 0}$ are
\begin{equation}
\mathbb{E}\big [e^{2\pi irZ_\tau}\big ] = 
\exp \Big \{\tau \int_{[0,1]}\frac{e^{2\pi ri\v } -1}{\v}\phi(d\v)
\Big \}.
\label{cts:08}
\end{equation}
\end{Proposition}
\begin{proof}
In the Poisson embedding
\begin{align*}
\mathbb{E}\big [e^{2\pi i r X_\tau}\mid X_0=x_0\big ]
&= \sum_{k=0}^\infty e^{-\lambda \tau} \frac{\big (\lambda \tau\big )^k}{k!}
\mathbb{E}\big [e^{2\pi i r\B_k}\mid \B_0=x_0\big ]
\nonumber \\
&=\sum_{k=0}^\infty e^{-\lambda \tau} \frac{\big (\lambda \tau\big )^k}{k!}
\mathbb{E}\big [e^{2\pi i r \V }\big ]^k\cdot e^{2\pi i r x_0}
\nonumber\\
&= 
\exp \big \{
\lambda\tau \mathbb{E}\big [e^{2\pi ir \V}-1\big ]\big \}
\cdot e^{2\pi i r x_0}.
\end{align*}
The subordinated process limit is straightforward.
\end{proof}
\begin{Example}
Take $\phi(\cdot)$ to be a Beta$(2-\gamma,\gamma)$ measure, where $\gamma \in (0,1)$.
\begin{align*}
\mathbb{E}\big [e^{iZ_\tau\vartheta }\big ] &= 
\exp \Big \{\tau \big (\Gamma(2-\gamma)\Gamma(\gamma)\big )^{-1}
\int_{[0,1]}\big (e^{i\v\vartheta } -1\big )\v^{-\gamma}(1-v)^{\gamma-1} d\v
\Big \}\\
&= \exp \Big \{\tau (1-\gamma)^{-1}\sum_{k=1}^\infty \frac{1}{k!}(i\vartheta)^k
\frac{(1-\gamma)_{(k)}}{1_{(k)}}\Big \}\\
&= \exp \Big \{\tau (1-\gamma)^{-1}\big (\phantom{\cdot}_1F_1(1-\gamma,1;i\vartheta) - 1\big )\Big \}{\color{green}.}
\end{align*}
\end{Example}
\medskip

\section{A Markov chain on ${\cal V}_{q,d}$}\label{MCVN}
Assume that transitions of entries of $X_t$ to $X_{t+1}$ are made according transition matrices $P(V[1]),\ldots, P(V[d])$ where $V=(V[1],\ldots V[d])$. The entries of $V$ are random on ${\cal V}_q$ and not necessarily independent. In the homogeneous random walk (\ref{rweqn:000}) a generic increment $Z\equiv V$ from (\ref{PVP:00}).
The probability of a transition $X_t=x$ to $X_{t+1}=y$ is then 
\begin{equation*}
P_{xy}=\mathbb{E}\big [\prod_{k=1}^d P_{x[k] y[k]}(V[k])\big ]
\end{equation*}
where expectation is over $V$. We use the notation (\ref{PV:00}).
The eigenvectors of the transition matrix $(P_{xy})$  are
\begin{equation}
\prod_{k=1}^d\theta_{r[k]}^{x[k]}=\theta_1^{x\dt r},
\label{ev:001}
\end{equation}
where $r=(r[1],\ldots, r[d])$, $r[k]\in {\cal V}_q$ and a complex conjugate form. The eigenvalues are 
\begin{equation}
\rho_r =   \mathbb{E}\big [\prod_{k=1}^d\theta_{r[k]}^{V[k]}\big ]
= \mathbb{E}\big [\theta_1^{V\dt r}\big ].
\label{eigenval:123}
\end{equation}
A Lancaster characterization of $P_{xy}^\circ$ with these eigenvectors is almost identical to the one-dimensional case.
It shows that all transition probability matrices with eigenvectors
(\ref{ev:001}) must have eigenvalues of the form (\ref{eigenval:123}).
\begin{Proposition}\label{pcirc}
Let $x,y \in  {\cal V}_{q,d}$. Then $P^\circ$ with entries 
\begin{align}
P^\circ_{xy} &= \frac{1}{q^d}
\sum_{r\in {\cal V}_{q,d}}\rho_r\prod_{k=1}^d
\theta_{r[k]}^{x[k]}\theta_{r[k]}^{-y[k]}
\label{spectral:02a}\\
&= \frac{1}{q^d}\sum_{r\in {\cal V}_{q,d}}\rho_r
\theta_1^{(x-y)\dt r}
\nonumber
\end{align}
is non-negative and therefore a transition probability matrix if and only if
$ \rho_r =  \mathbb{E}\big [\prod_{k=1}^d\theta_{r[k]}^{V[k]}\big ] = \mathbb{E}\big [\theta_1^{V\dt r}\big ]$ for a random variable $V$ on ${\cal V}_{q,d}$.
If $r$ has all entries zero then $\rho_r=1$.
\end{Proposition}
If $q=2$ Proposition \ref{pcirc} says that $(P^\circ_{xy})$ is a transition probability matrix with given eigenvectors if and only if $\rho_r=\mathbb{E}\big [(-1)^{V \dt r} \big ]$ and then
after simplification
\begin{equation*}
P_{xy}^\circ= \frac{1}{2^d}
\sum_{r\in {\cal V}_{q,d}}
\mathbb{E}\big [(-1)^{V \dt r}\big ]
(-1)^{(x+y)\dt r}.
\end{equation*}
This is effectively a special case of results in \cite{CG2021} written in a different notation. 

An analogue of Proposition \ref{pcirc:1} on a $d$-dimensional torus is the following.
\begin{Proposition} \label{torusprop} Let $(\B_t)_{t\in \mathbb{N}}$ be a homogeneous in time stochastic process on the torus group $\mathbb{T}^d$  with a uniform stationary distribution and eigenvectors $(e^{2\pi i  B_t\dt r})_{r\in \mathbb{N}^d}$. Then the class of such processes is characterized by having eigenvalues
\begin{equation*}
\rho_r = \mathbb{E}\big [e^{2\pi i\V\dt r}].
\end{equation*}
where $\V$ is a random variable on $[0,1]^d$.
\end{Proposition}
\subsection{Grouping terms when entries of $V$ are exchangeable}
\label{qexpansion:0}
\begin{Proposition}\label{qexpansion}
Suppose the entries of $V$ are exchangeable and the eigenvalues $(\kappa_l)$ are real.
Write $M_t = (M_{tj})_{j=0}^{q-1}$, where $M_{tj}$ counts the number of entries equal to $j$ in $X_t$.
The transition function for $M_t=m$ to $M_{t+1}=n$ has a spectral expansion
\begin{equation}
\mathbb{P}\big (n\mid m)=
\p(n)
\Bigg \{ 1 + \sum_{l:0<|l| \leq d}\kappa_lh_lQ_l(m;(\theta_k))\overline{Q_l(n;(\theta_k))}\Bigg \},
\label{spectral:001}
\end{equation}
where $\big (Q_l(\cdot ;(\theta_k))\big )$ are  multivariate Krawtchouk orthogonal polynomials on the uniform multinomial distribution
\[
\p(n) = {d\choose n}\frac{1}{q^d}.
\]
The state space of $(m)$ is ${\cal P}_{d,q}$, the set of partitions of $d$ into $q$ parts. \color{black}
$Q_l(m ;(\theta_k))$ is the coefficient of $w_1^{l_1}\cdots w_{q-1}^{l_q -1}$ in the generating function
\begin{equation}
G(m;w,q,d) =
\prod_{j=0}^{q-1}\Big (1 + \sum_{k=1}^{q-1}w_k\theta_k^j\Big )^{m[j]}
= \prod_{c=1}^d\Big (1 + \sum_{k=1}^{q-1}w_k\theta_k^{X[c]}\Big ),
\label{genfn:00}
\end{equation}
where $X$ is uniform  on ${\cal V}_{q,d}$.
$X$ and $m$ are random variables in this generating function.
$X[c]$ is random on the points ${\cal V}_q$  so $X$ is random on
${\cal V}_{q,d}$, independent in different entries $X[c]$. $m[j]$ is the multinomial count in $d$ trials of how many of $X[c]$ values over different $c$ are equal to 
$j \in {\cal V}_q$. This is the way the functions $Q_l(m ;(\theta_k))$ are set up to be orthogonal on the multinomial. 
In (\ref{genfn:00}) for $j\in {\cal V}_q$,
$m[j]= |\{c\in [d]: X[c]=j\}|$ let
\begin{equation*}
h_l^{-1} = \mathbb{E}\big [Q_l(M;(\theta_k))\overline{Q_l(M;(\theta_k))}\big ]
= \frac{d!}{(d-|l|)!l_1!\cdots l_{q-1}!},
\end{equation*}
where expectation is taken over $M$ having a uniform multinomial distribution
and $|l| = l_1+\cdots +l_{q-1}$.  $Q_l(m ;(\theta_k))$ is also the coefficient of ${d\choose m}s_0^{m[0]}\cdots s_{q-1}^{m[q-1]}$ in
\begin{equation}
h_l^{-1}\big (\sum_{j=0}^{q-1}s_j\big )^{d-|l|}
\prod_{k=1}^{q-1}\big (\sum_{j=0}^{q-1}s_j\theta_k^j\big )^{l_k},
\label{genfn:02}
\end{equation}
a dual generating function. Denote $l^+=(d-|l|,l_1,\ldots,l_{q-1})$ and 
$m^- = (m[1],\ldots, m[d-1])$. There is a duality relationship 
\begin{equation}
h_m^{-1}Q_l(m; (\theta_k)) = h_{l^+}^{-1}Q_{m^-}(l^+;(\theta_j)).
\label{MVKdual:00}
\end{equation} 
The eigenvalues 
$\kappa_l,\ j \in {\cal V}_q$ are a grouping of indices in $\rho_r$.
Let  $A_l = \{r[k],k\in[d]: |\{k:r[k]=j\}| = l_j, j \in {\cal V}_q\}$. Then 
\begin{align}
\kappa_l =  h_l \mathbb{E}\big [\prod_{A_l} \theta_{r[k]}^{V[k]}\big ]
=\mathbb{E}\big [
\theta_1^{S_{l_1}}
\theta_2^{S_{l_2}}
\cdots
\theta_{q-1}^{ S_{l_{q-1}} }\Big ]
\label{eval:26}
\end{align}
where
\begin{align*}
S_{l_1} &= V[1]+\cdots + V[l_1]\\
S_{l_2} &= V[l_1+1]+\cdots + V[l_1+1+l_2]\\
&\cdots\cdots\\
S_{l_{q-1}}&=V[l_1+\cdots + l_{q-2} +q-1] + \cdots +V[l_1+\cdots + l_{q-2} +q-1 + l_{q-1}].
\end{align*}
\end{Proposition}
\begin{proof}
In this exchangeable case (\ref{spectral:02a}) must be identical to (\ref{spectral:001}) from a probabilistic argument, once notation is established, apart from the factor ${d\choose n}$ which arises from unordering entries of $X_{t},X_{t+1}$.
The stationary distribution of $(M_t)$ is $\p(m)$.
 We find the spectral expansion for the joint distribution 
\[
\mathbb{P}\big (M_t = m, M_{t+1} = n \big )=
\p(n)\mathbb{P}\big (n\mid m),
\]
where $m$ and $n$ are the counts of values in  $X_{t},X_{t+1}$.  A computational way of showing (\ref{spectral:001}) is to consider expectation in the generating functions of $X_t,X_{t+1}$.
Let $X_t$ have a multinomial stationary distribution uniform on ${\cal V}_{q,d}$. 
\begin{align*}
&\mathbb{E}\Big [\overline{G(M_t;w,q,d)}{G(M_{t+1};w^\prime,q,d)}\Big ]\nonumber\\
&= \mathbb{E}\Big [\prod_{c=1}^d
\Big (1 + \sum_{k=1}^{q-1}w_k\theta_k^{-X_t[c]}\Big )
\Big (1 + \sum_{k=1}^{q-1}w^\prime_k\theta_k^{X_{t+1}[c]}\Big )
\Big ].
\end{align*}
Now conditional on $V=v$, $X_t[c],X_{t+1}[c]$ are independent and
\begin{align*}
\mathbb{E}\Big [
\Big (1 + \sum_{k=1}^{q-1}w_k\theta_k^{-X_t[c]}\Big )
\Big (1 + \sum_{k=1}^{q-1}w^\prime_k\theta_k^{X_{t+1}[c]}\Big )\mid V=v
\Big ]
= 1 + \sum_{k=1}^{q-1}w_kw^\prime_k\theta_k^{v[c]}.
\end{align*}
because 
\begin{align*}
\mathbb{E}\Big [\theta_k^{-X_t[c]}\theta_{k^\prime}^{X_{t+1}[c]}\mid V=v \Big ]
= \sum_{a,b=0}^{d-1}\frac{1}{q}P_{a,b}(v[c])\theta_k^{-a}\theta_{k^\prime}^{b}
= \delta_{k,k^\prime}\theta_k^{v[c]}.
\end{align*}
Taking the unconditional expectation over $V$ we have shown that
\begin{align*}
&\mathbb{E}\Big [\overline{G(M_t;w,q,d)}{G(M_{t+1};w^\prime,q,d)}\Big ]\nonumber\\
&= 1 + \sum_{l: |l|\leq q}h_l^{-1}\mathbb{E}\Big [\theta_1^{S_{l_1}}\theta_2^{S_{l_2}}\cdots \theta_{q-1}^{S_{l_{q-1}}}\Big ]
(w_1w_1^\prime)^{l_1}\cdots (w_{q-1}w_{q-1}^\prime)^{l_{q-1}}.
\end{align*}
Equating coefficients of powers of elements of $w,w^\prime$
\begin{equation}
\mathbb{E}\Big [\overline{Q_l(M_t;(\theta_k))}{Q_{l^\prime}(M_{t+1};(\theta_k))}\Big ]
= \delta_{l,l^\prime}h_l^{-1}\mathbb{E}\Big [\theta_1^{S_{l_1}}\theta_2^{S_{l_2}}\cdots \theta_{q-1}^{S_{l_{q-1}}}\Big ].
\label{eigenv:100}
\end{equation}
Denote the eigenvalues as $\kappa_l = \mathbb{E}\Big [\theta_1^{S_{l_1}}\theta_2^{S_{l_2}}\cdots \theta_{q-1}^{S_{l_{q-1}}}\Big ]$. \\[0.2cm]
Since the elements of $V$ are exchangeable the expectation can be arranged in consecutive blocks as in (\ref{eval:26}), this is the reason for the combinatorial term $h_l^{-1}$ in  (\ref{eigenv:100}). We have shown that the joint distribution of $M_t,M_{t+1}$ has the correct expansion, and therefore the conditional distribution (\ref{spectral:001}) holds.
 To obtain (\ref{genfn:02}), which is not needed for the proof of (\ref{spectral:001}) but is needed in what follows, consider
\begin{align}
&\sum_{m: |m| = d}{d\choose m}
\prod_{j=0}^{q-1}s_j^{m[j]}\Big (1 + \sum_{k=1}^{q-1}w_k\theta_k^j\Big )^{m[j]}\nonumber\\
&=\Big( \sum_{j=0}^{q-1}\big (1 + \sum_{k=1}^{q-1}w_k\theta_k^j\big )s_j\Big )^d\nonumber\\
&= \Big (\sum_{j=0}^{q-1}s_j + \sum_{k=1}^{q-1}\Big (\sum_{j=1}^{q-1}s_j\theta_k^j\Big )w_k\Big )^{d}.
\label{dual:002}
\end{align}
The coefficient of $\prod_{k=1}^{d-1}w_k^{l[k]}$ in (\ref{dual:002}) is (\ref{genfn:02}) showing the dual generating function is correct.
\end{proof}
General multivariate Krawtchouk polynomials $(Q_l)$, orthogonal on the multinomial were first introduced in \citet{G1971}. Probabilistic aspects of them are reviewed and developed in \citet{DG2014}.
\begin{Example}
Take $q=3, d=3$. $l=(1,2)$.\\
Considering a partition $A_l=\{r_k:|\{k:r_k=1\}|=1, |\{k:r_k=2\}|=2\}$ in (\ref{eval:26})
\begin{align*}
\kappa_l &= \mathbb{E}\big [
\theta_1^{V[1]}\theta_2^{V[2]+V[3]}\big ]
\\
&= h_l\mathbb{E}\big [Q_l(V;(\theta_k))\big ]\\
&=\frac{1!2!}{3!}\mathbb{E}\big [
\theta_1^{V[1]}\theta_2^{V[2]}\theta_2^{V[3]}
+\theta_2^{V[1]}\theta_1^{V[2]}\theta_2^{V[3]}
+\theta_2^{V[1]}\theta_2^{V[2]}\theta_1^{V[3]}
\big ].
\end{align*}
\begin{Remark}
  If $q=2$ the multivariate Krawtchouk polynomials are the Krawtchouk polynomials with notation $\big  (K_l(m[1];d,q)\big)_{l=0}^d$ in  the literature, orthogonal on the Binomial$(m[1];d,1-1/q)$ distribution with the generating function
  \begin{equation}
    \sum_{l=0}^d K_l(m;d,q)w^l=(1+(q-1)w)^{d-m}(1-w)^m.
    \label{genfn:S}
  \end{equation}
On the other hand, the generating function (\ref{genfn:00}) becomes 
\begin{equation*}
G(m;w,2,d) = (1+w[1])^{d-m[1]}(1-w[1])^{m[1]},
\end{equation*}
which shows that
\[
  K_l(m;d,2)=Q_l(m;(1,-1)),
\]
where we use notation $(1,-1)$ for the $2^{\text{th}}$ roots of unity.
\cite{DG2012} used a scaled version 
\[
Q^K_l(m;d,1-1/q) = {d\choose l}^{-1}(q-1)^{-l}
K_l(m;d,q),
\]
where note that $p$ and $q$ in \cite{DG2012} are $1-1/q$ and $1/q$, respectively, in this paper.
Then there is duality with $Q_l^K(m;d,1-1/q) = Q^K_m(l;d,1-1/q)$, the same as (\ref{MVKdual:00}) when $q=2$.
\end{Remark}

\end{Example}
\begin{Corollary}\label{nxycorr:01}
Let $n_{x-y}\in {\cal V}_{q,d}$ be the counts of entries\!\!\! $\mod q$ in $X_{t+1}-X_t$.
That is for $k\in {\cal V}_q$,
\[
n_{x-y}[k] = |\big \{j: X_{t+1}[j]- X_t[j]= k\!\!\!\mod q\big \}|.
\]
Then $n_{x-y}$ has distribution
\begin{equation}
\p(n_{x-y})
\Big\{ 1 + \sum_{l:0<|l| \leq d}\kappa_lQ_l(n_{x-y};(\theta_k))\Big \},
\label{spectral:006}
\end{equation}
which only depends on $x,y$ through $x-y \mod q$.
\end{Corollary}
\begin{proof}
The form of (\ref{spectral:02a}) can be expressed instead by a modification where $x[k]$ is replaced by 
$x[k]-y[k]\mod q$ and $y[k]$ is replaced by zero for $k \in {\cal V}_q$.
Consider this in the expansion (\ref{spectral:001}).
When all the counts in $y$ are zero, $n_{y}[0]=d$ and
\[
h_l\overline{Q_l\big ((d,0,\ldots, 0);(\theta_k)\big )} = h_l\frac{d!}{(
d-|l|)!l_1!\cdots l_{q-1}!} = 1.
\]
Substituting the modified functions into (\ref{spectral:001}) gives (\ref{spectral:006}).
\end{proof}
\begin{Lemma}\label{QtoKCorr}
\begin{equation}
\mathbb{E}\big [Q_l(M;(\theta_k))\mid M[0]=m[0]\big ] 
= (q - 1)^{-|l|}{|l|\choose l}
  K_{|l|}(d-m[0];d,q).
\label{QtoK:00a}
\end{equation}
\end{Lemma}
\begin{proof}
\begin{align}
&\sum_{m:m[0]\text{~fixed}}{d\choose m}\Big (\frac{1}{q}\Big )^d G(m;w,q,d)
\nonumber\\
&= {d\choose m[0]}\Big (\frac{1}{q}\Big )^d\Big (1 + \sum_{k=1}^{q-1}w_k\Big )^{m[0]}
\Big (\sum_{j=1}^{q-1}\Big (1 + \sum_{k=1}^{q-1}w_k\theta_k^j\big )\Big )^{d-m[0]}
\nonumber\\
&= {d\choose m[0]}\Big (\frac{1}{q}\Big )^{m[0]}\Big (1 - \frac{1}{q}\Big )^{d-m[0]}
\Big (1 + \sum_{k=1}^{q-1}w_k\Big )^{m[0]}
\Big (1 - \frac{1}{q-1}\sum_{k=1}^{q-1}w_k\Big )^{d-m[0]}
\label{QtoK:01}
\end{align}
Identifying (\ref{QtoK:01}) with the generating function (\ref{genfn:S}) and finding coefficients
of $w_1^{l_1}\cdots w_{q-1}^{l_{q-1}}$ shows (\ref{QtoK:00a}).
\end{proof}

\begin{Remark}\label{sl1}
The multivariate Krawtchouk polynomials discussed in this
paper are a special case of those discussed in \citet{DG2014}.
The complete set of orthogonal functions 
\begin{equation}
  (u_0^{(k)},u_1^{(k)},u_2^{(k)},\ldots,u_{q-1}^{(k)}), \quad
  k\in\mathcal{V}_q
  \label{seq1}
\end{equation}
is now taken to be
\[
  (1,\theta_k^1,\theta_k^2,\ldots,\theta_k^{q-1}), \quad
  k\in\mathcal{V}_q
\]
with  orthogonality
\[
  \frac{1}{q}\sum_{j=0}^{q-1} u^{(k)}_j \overline{u^{(l)}_j}
  =\frac{1}{q}\sum_{j=0}^{q-1} \theta^j_k \overline{\theta^j_l}
  =\delta_{kl}.
\]
The development of the multivariate Krawtchouk polynomials with this basis is in
\cite{M2004}, where the multivariate Krawtchouk polynomials were identified as the Aomoto-Gelfand hypergeometric polynomials. The generalized version, where \eqref{seq1} is determined by character algebras, was presented in \cite{MT2004}. 
 The duality relationship (\ref{MVKdual:00})
is given as Equation (33) in \citet{DG2014}.
\end{Remark}
The normalized version of the multivariate Krawtchouk polynomials
discussed here are the zonal spherical functions
of the Gelfand pair $(K\wr S_d,L\wr S_d)$, where $(K,L)=(\mathbb{Z}_q,\{0\})$ is again a Gelfand pair, and $S_d$ is the symmetric group. Here, the character algebra in Remark~\ref{sl1} is the Hecke algebra of $(\mathbb{Z}_q,\{0\})$.
The zonal spherical functions are \cite{M2004}
\[
  \omega_{l+}^m=h_m m_{\lambda(m)}(\underbrace{1,\ldots,1}_{l_0},\underbrace{\theta_1,\ldots,\theta_1}_{l_1},\cdots,\underbrace{\theta_{q-1},\ldots,\theta_{q-1}}_{l_{q-1}})=h_m Q_l(m;(\theta_k)),
\]
where $\lambda(m)$ is the partition $0^{m_0}1^{m_1}\cdots(q-1)^{m_{q-1}}$
and $m_{\lambda(m)}$ is the monomial symmetric function.
For example, if $q=3$, $d=4$, $l=(2,1)$, and $m=(1,1,2)$,
$\lambda(m)=(0^11^12^2)$ and we can confirm that
\[
Q_{(2,1)}((1,1,2);(\theta_k))=m_{1^12^2}(1,\theta_1,\theta_1,\theta_2)=
-3e^{4\pi i/3}.
\]
An element $g=(g_1,\ldots,g_{d};\sigma)\in \mathbb{Z}_q\wr S_d$ acts on
$\mathcal{V}_{q,d}$ as
\[
  (g_1,\ldots,g_{d};\sigma)(a_1,\ldots,a_{d})=
  (g_1+a_{\sigma^{-1}(1)},\ldots,g_{d}+a_{\sigma^{-1}(d)})
\]
and the stochastic matrix is invariant under the group action,
because
\[
  P^{\otimes d}(ga,gb)=\prod_{k=1}^{d}P(a_k+g_{\sigma^{-1}(k)},b_k+g_{\sigma^{-1}(k)})
  =\prod_{k=1}^{d}P(a_k,b_k)=P^{\otimes d}(a,b),
\]
where the second equality holds because $P$ is a circulant matrix.
It is known that we have a spectral expansion in terms of the zonal
spherical functions.  
\begin{Example}\label{Mexample}\
  Spectral expansions of multitype random walks of generic circulant random walks demand the zonal spherical functions of the Gelfand pair $(\mathbb{Z}_q\wr S_d,S_d)$ as just we described. Mizukawa \cite{M2018} considered spectral expansions of multitype random walks, based on Gelfand pairs $(K\wr S_d,L\wr S_d)$ with some Gelfand pairs $(K,L)$ determined by the stochastic matrix as a function on $(K/L)^d\times(K/L)^d$, which are reported in  \cite{DG2014}. These are special cases of circulant random walks discussed here. 
 \begin{itemize}
\item
  Example 2.2(2), \cite{M2018} with $p=1$. 
  Pick a coordinate with probability $1/d$  and move one unit to the left in $\mathcal{V}_q$. This is an extension of Example \ref{Example:1}  and an ergodic random walk.
  The eigenvalues have the form
  \[
  \kappa_l=\frac{1}{d}\sum_{j=0}^{q-1}l_j\theta_1^{-j}.
  \]
  We need the zonal spherical functions of the Gelfand pair $(K\wr S_d,L\wr S_d)=(\mathbb{Z}_q\wr S_d,S_d)$.
\item  Example 2.2(3), \cite{M2018}.
  With probability $1-\gamma$, the random walk does not move. If it moves,
  pick
  a coordinate with probability $1/d$ and move to a left or right neighbour.
  This is an extension of Example \ref{Example:2}.
  It is ergodic. This is not a symmetric
  random walk so we need multivariate Krawtchouk polynomials.
  The eigenvalues have the form
  \[
  \kappa_l=(1-\gamma)+\frac{\gamma}{2d}\sum_{j=0}^{q-1}l_j(\theta_1^{-j}+\theta_1^{j}) 
  =(1-\gamma)+ \frac{\gamma}{d}\sum_{j=0}^{q-1}l_j\cos(2\pi/q\cdot j).
  \]
  This process is reversible because the eigenvalues are real.
  We may take $(K,L)=(D_q,\langle b\rangle)$, where $D_q=\{\langle a,b\rangle|a^q=b^2=(ab)^2=1\}$ is the dihedral group. The double coset $L\backslash K/L=\mathbb{Z}_q$ is identical to that of the above example. The zonal spherical function of $(D_q\wr S_d,\langle b\rangle\wr S_d)$ is given by taking the complete set of orthogonal functions
  \[
  (1,\cos (2\pi k/q\cdot 1),\ldots,\cos(2\pi k/q\cdot(q-1))), \quad
  k\in\mathcal{V}_q,
  \]
  or the zonal spherical functions of $(D_q,\langle b\rangle)$, for \eqref{seq1}. See \cite{AM2003}.
  \item The non-lazy version of Example 2.2(1), \cite{M2018} ($p=1/(r-1)$).
  Pick a coordinate with $1/d$  and pick a position from $\mathcal{V}_q$ except for the current position. This is precisely the same random walk as Example 2.1 in \cite{GM2025}.
  We may take $(K,L)=(S_q,S_{q-1})$. A complete representative set of the double coset  $L\backslash K/L$ is $\{e_K,(1,q)\}$, where $(1,q)\in S_q$ is a transposition. The double coset is two-dimensional and parametrized by $(k_0,k_1=d-k_0)$, and the zonal spherical functions of $(S_q,S_{q-1})$ are $\omega_0=1$, $\omega_1(x_0)=-1$ and $\omega_1(x_1)=1$. Then, the zonal spherical functions of $(S_q\wr S_d,S_{q-1}\wr S_d)$ are the univariate Krawtchouk polynomials (see Proposition 2.4, \cite{M2018}). A probabilistic argument why we do not need multivariate Krawtchouk polynomials will be given in Remark~\ref{Hamming:23}.
\end{itemize}
\end{Example}
\begin{Remark}
If $q=2$ the Multivariate Krawtchouk polynomials are the Krawtchouk polynomials $\big (K_j(m[1];d,1/2)\big )_{j=0}^d$, orthogonal on the Binomial$(m[1];d,1/2)$ distribution. The generating function (\ref{genfn:00}) becomes 
\begin{equation*}
G(m;w,2,d) = (1+w[1])^{d-m[1]}(1-w[1])^{m[1]}.
\end{equation*}
\cite{DG2012} use a scaled version 
\[
Q^K_j(m[1];d,1/2) = {d\choose j}^{-1}Q_j(m[1];d,1/2). 
\]
Then there is duality
with $Q_j^K(m[1];d,1/2) = Q^K_{m[1]}(j;d,1/2)$, the same as (\ref{MVKdual:00}) when $q=2$.
\end{Remark}
\begin{Remark}
Let
\begin{equation}
Q_{|l|}(n,m) = \sum_{l:|l|\text{~fixed}}h_lQ_l(m;(\theta_k))\overline{Q_l(n;(\theta_k))}.
\label{RKP:88}
\end{equation}
The polynomials (\ref{RKP:88}) in two variables are called reproducing kernel polynomials on the multinomial. 
These important polynomials do not depend on the basis used for the polynomials in their sum, so do not depend on $\big (\theta_1^r\big )$ in our case. They are studied in \cite{DG2019} and have an explicit form. 
Denote the multinomial distribution as
\[
\p(n) = {|n|\choose n}\big (\frac{1}{q}\big )^{|n|},\ |n|\text{~fixed}.
\]
Then
\begin{equation}
Q_{|l|}(n,m)\p(n)\p(m) = \sum_{k=0}^{|l|}(-1)^{|l|-k}{d\choose k}{d-k\choose |l|-k}\p(n)\p(m)\zeta_k(n,m),
\label{RKP:333}
\end{equation}
where 
\[
\p(n)\p(m)\zeta_k(n,m) = \sum_{w:|w|=k}\mathbb{P}\big (n\mid w)\mathbb{P}\big (m\mid w)\mathbb{P}\big (w)
\]
with 
$
\mathbb{P}\big (w) = \p(w),\ \mathbb{P}\big (n\mid w) = \p(n-w),\ \mathbb{P}\big (m\mid w)= \p(m-w).
$
The formula (\ref{RKP:333}) comes from Remark 7 in \cite{DG2019}.
$\p(n)\p(m)\zeta_k(n,m)$ has a probability interpretation of obtaining a configuration of $(n,m)$ when a sample of $|w|$ is taken in common for the first observations of $|n|=d,|m|=d$, then the remanding $|n|-|w|, |m|-|w|$ are chosen independently. The joint \pgf of $\p(n)\p(m)\zeta_k(n,m)$ is
\begin{equation*}
G_k(s,s^\prime)=
\Big (\frac{1}{q}\Big )^{2d-k}\Big (\sum_{j=0}^{q-1}s_js^\prime_j\Big )^k
\Big (\sum_{j=0}^{q-1}s_j\Big )^{d-k}
\Big (\sum_{j=0}^{q-1}s^\prime_j\Big )^{d-k}.
\end{equation*} 
\end{Remark}
\medskip

It is useful to have an estimate of the multivariate Krawtchouk polynomials as $d \to \infty$. The next Lemma gives a simple estimate. A more sophisticated central limit theorem that follows is in Lemma \ref{QCLT}.
\begin{Lemma}\label{Kplimit}
Take $n^{(d)}$ such that
 $d^{-1}n^{(d)}[j]\to z[j]$, $j=1,\ldots, q-1$, so $d^{-1}n^{(d)}[0] \to 1 - |z|$. 
Then
\begin{equation}
\lim_{d \to \infty}
h_lQ_l(n^{(d)};(\theta_k))
= \prod_{k=1}^{q-1}\Big ( 1 - |z| + \sum_{j=1}^{q-1}z[j]\theta_k^j\Big )^{l_k}.
\label{asymptotic:99}
\end{equation}
If $q=2$, $l$ has dimension one, and
\begin{equation}
\lim_{d \to \infty}
h_{l_1}Q_{l[1]}(n^{(d)};(\theta_k))
= \Big ( 1 - 2z[1] \Big )^{l_1}.
\label{asymptotic:99a}
\end{equation}

\end{Lemma}
\begin{proof}
Consider the limit in the generating function where $w$ is replaced by $d^{-1}w$.
\begin{align*}
&\big (1 + d^{-1}\sum_{k=1}^{q-1}w_k\big )^{d - \sum_{j=1}^{q-1}n^{(d)}[j]}
\prod_{j=1}^{q-1}\big (1 + d^{-1}\sum_{k=1}^{q-1}w_k\theta_k^j\big )^{n^{(d)}[j]}\\
&\to \exp \Big \{
(1-|z|)\sum_{k=1}^{q-1}w_k + \sum_{j=1}^{q-1}z[j]\sum_{k=1}^{q-1}w_k\theta_k^j\Big \}.
\end{align*}
Identifying coefficients of $\prod_{k=1}^{q-1}w_k^{l_k}$
\[
d^{-|l|}Q_l(n^{(d)};(\theta_k)) \to \prod_{k=1}^{q-1}\frac{1}{l_k!}
\Big ( 1 - |z| + \sum_{j=1}^{q-1}z[j]\theta_k^j\Big )^{l_k},
\]
which is equivalent to the statement of the Lemma  because $h_l = {d\choose l}^{-1}$ and
\[
\frac{h_ld^{|l|}}{\prod_{k=1}^{q-1}l[k]!} \to 1.
\]
\end{proof}
\begin{Lemma}\label{RKP:001}
In the reproducing kernel polynomials $Q_{|l|}(n^{(d)},m^{(d)})$ take 
\[
d^{-1}n^{(d)}[j] \to \xi[j],\ d^{-1}m^{(d)}[j] \to \eta[j],\ j=1,\ldots, q-1.
\]
Then
\begin{equation*}
\lim_{d\to \infty}
{d\choose |l|}^{-1}Q_{|l|}(n^{(d)},m^{(d)}) =
\Big (q(1-|\xi|)(1-|\eta|) + q\sum_{j=1}^{q-1}\xi[j]\eta[j]-1\Big )^{|l|}.
\label{QQ:77}
\end{equation*}
\end{Lemma}
\begin{proof}
Using Lemma \ref{Kplimit} to find the limit
\begin{align*}
&\lim_{d\to \infty}
{d\choose |l|}^{-1}Q_{|l|}(n^{(d)},m^{(d)})\\
&= \lim_{d\to \infty}\sum_{l:|l|\text{~fixed}}
{d\choose |l|}^{-1}h_l^{-1}h_lQ_l(m;(\theta_k))h_l\overline{Q_l(n;(\theta_k))}\\
&= \lim_{d\to \infty}\sum_{l:|l|\text{~fixed}}
{|l|\choose l}h_lQ_l(m;(\theta_k))h_l\overline{Q_l(n;(\theta_k))}\\
&=\Big (q(1-|\xi|)(1-|\eta|) + q\sum_{j=1}^{q-1}\xi[j]\eta[j]-1\Big )^{|l|}
\end{align*}
\end{proof}
\begin{Remark}
The chi-squared distance between $\mathbb{P}_t\big (n\mid m\big )$, the $t$-step transition probability corresponding to (\ref{spectral:001}) in Proposition \ref{qexpansion} and the stationary distribution $\p(n)$ is equal to
\begin{align}
\chi^2_t(m;d) &= \sum_{n \in {\cal P}_{d,q}}
\frac{\big |\mathbb{P}_t\big (n\mid m\big )-\p(n)\big |^2}{\p(n)}
\nonumber\\
&= \sum_{l\in  {\cal P}_{d,q}}|\kappa_l|^{2t}h_l|Q_l(m)|^2
\label{chisq:001}
\end{align}
If $\kappa_l$ only depends on $|l|$ then $\chi^2_t(m;d)=\sum_{|l|=0}^d|\kappa_l|^{2t}Q_{|l|}(m,m)$.
In (\ref{chisq:001}) $\kappa_l$ depends on $d$ and we are interested in an asymptotic form for 
$\chi^2_t(m;d)$ as $d\to \infty$ and $t\to \infty$. Adopting notation and  knowing (\ref{asymptotic:99}) in Lemma \ref{Kplimit}, as $d\to \infty$
\begin{align}
\chi^2_t(m;d) &\sim \sum_{l\in  {\cal P}_{d,q}}|\kappa_l|^{2t}
h_l^{-1}\prod_{k=1}^{q-1}\Big |1 - |z| + \sum_{j=1}^{q-1}z[j]\theta_k^j\Big |^{2l_k}\nonumber\\
&\leq \sum_{l\in  {\cal P}_{d,q}}|\kappa_l|^{2t}h_l^{-1}
\label{chisq:002}
\end{align}
At $m_0$ when $m[0]=d$, $m[j]=0$, $j=1,\ldots, q-1$, since $Q_l\big (m_0; (\theta_j)\big )=h_l^{-1}$ the exact value of chi-squared is 
\begin{equation*}
\chi^2_t(m_0;d)=
\sum_{l\in  {\cal P}_{q,d}}|\kappa_l|^{2t}h_l^{-1},
\end{equation*}
identical to (\ref{chisq:002}). 
These chi-squared bounds lead to cutoff estimates, as in the next Example \ref{simple:000}.
\end{Remark}
\begin{Example}\label{simple:000}
A model that is developed in \cite{CG2021} when $q=2$ is to choose uniformly a random subset of size $z_q$ from $X_t$ in (\ref{rweqn:000}) and then toggle the identity of these chosen variables. 
This is a randomized version of Example 2.3 in \cite{GM2025}.
The cutoff time is found in \cite{CG2021} when $d \to \infty$. We consider an extension of this model to a random walk on ${\cal V}_{q,d}$. Choose a subset of $\A$ variables in $X_t$ to change to form $X_{t+1}$. In an entry $X_t[j]$ that is changed $X_t[j] \to X_t[j] + V$ where $V$ is uniform on $1,2,\ldots, q-1$. $(X_t)$ is an ergodic Markov chain. In calculating the eigenvalues $\kappa_l$ in (\ref{spectral:001}), what is important is $|l|$.
If $r[j]>0$ and $X_t[j]$ is chosen to change then $\sum_{v\ne 0}e^{2\pi i r[j]v}= -1$. The other product terms in the eigenvalues are unity, therefore
\begin{equation*}
\kappa_l = \sum_{n=0}^{|l|}(-1)^n
\frac{
{|l|\choose n}{d-|l|\choose \A-n}
}
{
{|l|\choose \A}
}
= {d\choose |l|}^{-1}Q^K_{|l|}(\A;d,1/2).
\end{equation*}
As an extension we could choose $\A$ to be random. 
Simplifying
\begin{equation*}
\mathbb{P}\big (n\mid m) =
\p(n)\Big \{1 + \sum_{|l|=1}^\A {d\choose |l|}^{-1}Q^K_{|l|}(\A;d,1/2)Q_{|l|}(n,m) \Big \},
\end{equation*}
where $Q_{|l|}(n,m)$ are reproducing kernel polynomials. Suppose that $\A/d = \K\leq 1/2$ as $d \to \infty$. 
The chi-squared estimate between $\mathbb{P}_t\big (\cdot \mid m)$ and the stationary distribution $\p(\cdot)$ in this model is
\begin{align*}
\chi^2_t(m;d) 
&= \sum_{|l|=1}^\A\kappa_{|l|}^{2t}|Q_{|l|}(m,m)|^2\nonumber \\
&\sim \sum_{|l|=1}^{\A}{d\choose |l|} (1-2\K)^{2|l|t}{d\choose |l|}^{-1}Q_{|l|}(m,m)\nonumber \\
&< \Big (1 + (1-2\K)^{2t}\Big (q(1-|z|)^2 + q\sum_{j=1}^{q-1}z[j]^2-1\Big )\Big )^d - 1,
\end{align*}
where we used
\[
  \mathbb{E}\Big [Q_{|l|}(N,m)\overline{Q_{|l^\prime|}(N,m)}\Big ]  
  =Q_{|l|}(m,m)\delta_{|l|,|l^\prime|}\ge 0.    
\]
Lemma \ref{Kplimit}, (\ref{asymptotic:99a}) gives
\[
{d\choose |l|}^{-1}Q^K_{|l|}(\A,d,1/2) \to (1-2\K)^{2|l|}.
\]
If $\K=1/2$, $\chi^2_t(m;d)\sim 0$, implying mixing in 1-step. Suppose $\K < 1/2$.\\

The maximum chi-squared value occurs when all entries of $z$ are zero, then\\
\begin{align}
\max_{m \in {\cal V}_{q,d}}\chi^2_t(m;d) &= 
\Big (1 + (1-2\K)^{2t}(q-1)\Big )^d -1\nonumber\\
&\leq e^{d (q-1)(1-2\K)^{2t}} - 1\nonumber\\
&\leq e^{ d (q-1)e^{-4\K t}}-1
\label{Upper:00}
\end{align}
Choose 
\begin{equation*}
t_C= \frac{1}{4\K}\Big (\log\big (d(q-1)\big )+ C\Big ),
\end{equation*}
where $C$ is a constant. When $t=t_C$ the upper bound (\ref{Upper:00}) is equal to
\begin{equation*}
e^{e^{-C}}-1.
\end{equation*}
For a lower bound 
\begin{align*}
\max_{m \in {\cal V}_{q,d}}\chi^2_t(m;d) 
\geq\chi^2_t(m_0;d)
\geq d(q-1)(1-2\K)^{2t},
\end{align*}
because 
\begin{align*}
\chi_t^2(m_0,d)&= \sum_{|l|=1}^\A|\kappa_{|l|}|^{2t}|Q_{|l|}(m_0,m_0)|^2\\
&=\sum_{|l|=1}^\A|\kappa_{|l|}|^{2t}{d\choose |l|}(q-1)^{|l|}\geq d(q-1)(1-2\K)^{2t}.
\end{align*}
Then
\[
d(q-1)(1-2\K)^{2t_C} \sim d(q-1)e^{-4\K t_C} = e^{-C}.
\]
so the limit bound is
\begin{equation*}
\max_{m\in {\cal P}_{dq}}\chi^2_t(m;d) \geq e^{-C}.
\end{equation*}
for small $\K$.
Considering $C>0$ or $C<0$ shows that 
\begin{equation*}
t_{\text{cutoff}} = \frac{1}{4\K}\log(d(q-1))
\end{equation*}
is the cutoff time.
 For large positive $C$, both upper and lower bounds become small, while for small negative $C$, both upper and lower bounds become large. This cutoff reproduces Theorem 4 of \cite{CG2021} when $q=2$. Moreover when $\K=1/d$ it is a well known result.
\end{Example}
\begin{Remark}\label{Cayley}
  Consider a Cayley graph of a finite group $G$ with $k$ generators. If transitions occur by a generator chosen uniformly at random, we call it a random Cayley graph. If $1\ll \log k \ll \log |G|$, it has been established in \cite{HOT2026} that the random walk exhibits cutoff for all Abelian groups if $k\gg\log |G|$. Moreover, \cite{HOT2026}  recently extended this result; they established the cutoff if $k-m(G)\gg 1$, where $m(G)$ is the minimal size of a generating subset of $G$. In Example~\ref{simple:000}, $\log|G|=\log|\mathbb{Z}_q^d|=d\log q$ and the number of generators (possible moves) is $k={d\choose \A}(q-1)^\A$, and each move has the same probability if $\A$ is a fixed value almost surely. Since $1\ll \log k \ll \log |G|$ and $k\gg\log |G|$ for large $d$, the random walk exhibits cutoff, as expected.
\end{Remark}  
\begin{Remark}
Let $U$ be Bernoulli$(1/2)$. An orthogonal set of functions on $U$ is
$\{u^{(0)}=1,u^{(1)}=1-2U\}$. A hypergroup property says that for $a,b,j\in \{0,1\}$ 
\begin{equation}
u^{(a)}u^{(b)}= \sum_{j=0}^1c(a,b,j)u^{(j)}, 
\label{hypergroup:00}
\end{equation}
where $c(a,b,j) \geq 0$.
In this particular case the hypergroup property is almost trivial because it is clearly true if one or more of $a,b$ are zero and
${u^{(1)}}^2 = 1 = u^{(0)}$ for $U=0\text{~or~}1$. In the complex space of orthogonal functions $\{\theta_k^a\}, k,a,v \in {\cal V}_d$ the analogue of (\ref{hypergroup:00}) is 
\begin{equation*}
\theta_r^a\theta_r^{-b} = \sum_{v=0}^{q-1}c(a,b,v)\theta_r^{-v},\ a,b \in {\cal V}_q.
\end{equation*}
where $c(a,b,v) \geq 0$. In this context the hypergroup property is also easy because
\[
\theta_r^a\theta_r^{-b} = \theta_r^{-b+a\!\! \mod q}.
\]
The hypergroup property holds for $\big (Q_l(m;(\theta_k)\big )$ because of the form of the functions. That is
\begin{equation}
h_lQ_l\big (m;(\theta_k)\big )\overline{Q_l\big (n;(\theta_k)\big )}
= \sum_{g \in {\cal P}_{d,q}} c(m,n,g)\overline{Q_l\big (g;(\theta_k)\big )},
\label{hypergroup:02}
\end{equation}
for $c(m,n,g)\geq 0$. the hypergroup property implies (\ref{hypergroup:02}).\\ We fix normalization 
considering $m_0,n_0,g_0$ having all entries indexed by $j\geq 1$ zero, so that
\begin{equation}
\frac{Q_l\big (m;(\theta_k)\big )}{Q_l\big (m_0;(\theta_k)\big )}
\frac{\overline{Q_l\big (n;(\theta_k)\big )}}{\overline{Q_l\big (n_0;(\theta_k)\big )}}
= \sum_{g \in {\cal P}_{d,q}} c(m,n,g)
\frac{\overline{Q_l\big (g;(\theta_k)\big )}}{\overline{Q_l\big (g_0;(\theta_k)\big )}},
\label{hypergroup:02a}
\end{equation}
fixing $c(m_0,n_0,g_0)=1$. The denominators in (\ref{hypergroup:02a}) are all $h_l^{-1}$ so 
(\ref{hypergroup:02}) holds.
Using the dual orthogonal functions
knowing in the dual that
\[
\sum_{l \in {\cal P}_{d,q}}h_lQ_l(g;(\theta_k))\overline{Q_l(g;(\theta_k))} =  \p(g)^{-1}
\]
gives
\begin{equation}
\sum_{l \in {\cal P}_{d,q}}\Big (Q_l(m;(\theta_k))\overline{Q_l(n;(\theta_k))}\Big )h^2_lQ_l(g;(\theta_k))
= c(m,n,g) \p(g)^{-1} \geq 0.
\label{threeseries:00}
\end{equation}
Hypergroup properties are implicit from Propositions \ref{pcirc} and 
\ref{qexpansion} by taking the entries in $V$ to be fixed.
\cite{E1969,DG2012,DG2014} detail hypergroup properties in the Krawtchouk and 
Multivariate Krawtchouk polynomials. 
\end{Remark}
\begin{Proposition}\label{Generaltrans}
An expansion (\ref{spectral:001}) with given eigenfunctions 
$\big (Q_l(m;(\theta_k)\big )$ is a transition function, if and only if
\begin{equation*}
\kappa_l=h_l\mathbb{E}\big [Q_l(U;(\theta_k))\big ].
\end{equation*}
Then (\ref{spectral:001}) can be written as 
\begin{equation}
\p(n)\mathbb{E}\bigg [
 1 + \sum_{l:|l| \leq d}h_l^2Q_l(U;(\theta_k))Q_l(m;(\theta_k))\overline{Q_l(n;(\theta_k))}\bigg ],
\label{spectral:0011}
\end{equation}
with expectation over $U$. 
\end{Proposition}
\begin{proof}
The sufficiency follows because of (\ref{threeseries:00}). \\
For the necessity suppose that (\ref{spectral:001}) holds.
Then
\begin{equation}
\mathbb{E}\big [Q_l(n;(\theta_k))\mid m\big ] = \kappa_lQ_l(m;(\theta_k)).
\label{cond:222}
\end{equation}
Let $m=m_0$ with all zero entries in (\ref{cond:222}). Then $Q_l(m_0;(\theta_k))=h_l^{-1}$ so
\[
\kappa_lh_l^{-1} = \mathbb{E}\big [Q_l(n;(\theta_k))\mid m_0\big ],
\]
completing the necessity proof.
\end{proof}
\begin{Remark}
In view of Lemma \ref{QtoKCorr} another sum similar to (\ref{spectral:0011}) with $Q_l(U;(\theta_k))$ replaced by 
\begin{equation*}
(q - 1)^{-|l|}{|l|\choose l}
  K_{|l|}(d-u[0];d,q)
\label{QtoK:00}
\end{equation*}
is also a transition function. Then (\ref{spectral:0011}) simplifies to a series in the reproducing Kernel polynomials
\begin{align}
&\p(n)\Big \{1 + \sum_{l:|l|\leq d}
(q - 1)^{-|l|}{|l|\choose l}
  K_{|l|}(d-u[0];d,q)
  h_l^2Q_l(m;(\theta_k))\overline{Q_l(n;(\theta_k))} \Big \}\nonumber\\
  &=\p(n)\Big \{1 + \sum_{|l|=0}^d(q - 1)^{-|l|}{d \choose |l|}^{-1}
  K_{|l|}(d-u[0];d,q)Q_{|l|}(m,n) \Big \}.
  \label{Qidentity:00}
\end{align}
If $v[0]=d$ then (\ref{Qidentity:00}) is equal to 
\[
\p(n)\Big \{1 + \sum_{|l|=0}^d
  Q_{|l|}(m,n) \Big \} = 1.
\]
\end{Remark}
\begin{Remark}
The extreme points of the convex set of distributions (\ref{spectral:0011}) occur when $U$ is non-random. However our interest sometimes is in transition densities when $\mathbb{E}\big [Q_l(U;(\theta_k)\big ]$ is real.
The extreme points are then when entries of $X$ independently take fixed values $x[c]$ or $q-x[c]$ in \eqref{genfn:00}.
 Let $u[j]$ be the count of the number of $x[c]$ values which are $j$.
The generating function for these extreme points is
\begin{equation}
G_{\cal R}(u;w,q,d)=
\prod_{j=0}^{q-1}\Big (1 + \sum_{k=1}^{q-1}w_k\cos(2\pi/q\cdot jk)\Big )^{u[j]}.
\label{extremesymm:00}
\end{equation}
The functions, coefficients of $w_1^{l_1}\cdots w_{q-1}^{l_{q}-1}$,  will be denoted by $Q_l^{\cal R}(u)$.
To obtain the generating function, from (\ref{genfn:00}),
\begin{align*}
G_{\cal R}(u;w,q,d)&=
\prod_{c=1}^d\Big [
\frac{1}{2}\Big (1 + \sum_{k=1}^{q-1}w_k\theta_k^{x[c]}\Big )
+ \frac{1}{2}\Big (1 + \sum_{k=1}^{q-1}w_k\theta_k^{q-x[c]}\Big )\Big ]\nonumber\\
&=\prod_{c=1}^d\Big [1 + \sum_{k=1}^{q-1}w_k\cos (2\pi/q\dt x[c]k)\Big ]\nonumber\\
&=\prod_{j=0}^{q-1}\Big (1 + \sum_{k=1}^{q-1}w_k\cos(2\pi/q\dt  jk)\Big )^{u[j]}.
\end{align*}
\end{Remark}
\begin{Example}\label{Example:2a}
In the model of Example \ref{Example:2}, $p[1]=\gamma/2, p[q-1]=\gamma/2, p[0] = 1-\gamma$.
In a de Finetti model $\gamma$ is a random variable.
Then
\[
1 + \sum_{k=1}^{q-1}w_k
\sum_{j=0}^{q-1}p[j]\theta_k^j
= 1 + \sum_{k=1}^{q-1}w_k\Big (1 - \gamma(1-\cos(2\pi k/q))\Big ).
\]
so
\[
\kappa_l = h_l\mathbb{E}\Big [\prod_{k=1}^{q-1}\Big (1 - \gamma(1-\cos(2\pi k/q))\Big )^{l_k}\Big ].
\]
\end{Example}
\begin{Example}\label{bigchi}
Example \ref{simple:000} is extended by finding the chi-squared distance between between 
$\mathbb{P}_t(\cdot \mid m)$ and $\p(\cdot)$ when the transition density of the process is an extreme point in the convex 
set of transition densities (\ref{spectral:0011}) in Proposition \ref{Generaltrans} when $u$ is not random, apart from 
assuming $Q^{\cal R}_l(u)$ is generated by (\ref{extremesymm:00}). 
In this model jumps are made with each entry of $v$ such that $v[k]=1$ or $v[k]=q-1$ with probability $1/2$, $k\in [d]$, so the jumps are symmetric. The random walk is ergodic if $q$ is prime and $u^{(d)}[0]=0$. 
Assume that 
$m\equiv m^{(d)}, u\equiv u^{(d)}$ and
\[
  d^{-1}m^{(d)}[j] \to \xi[j],\
  d^{-1}u^{(d)}[j] \to \gamma[j],\ j=1,\ldots, q-1,
\]
where $\gamma[j]=\gamma[q-j]$.
The chi-squared estimate between $\mathbb{P}_t\big (\cdot \mid m)$ and the stationary distribution $\p(\cdot)$ in this model is
\begin{align*}
\chi^2_t(m;d) 
&= \sum_{l\in  {\cal P}_{d,q}, |l|\geq 1}|\kappa_l|^{2t}h_l|Q_l(m;(\theta_k))|^2\nonumber \\
&=\sum_{l\in  {\cal P}_{d,q}, |l|\geq 1}|h_lQ^{\cal R}_l(u)|^{2t}h_l|Q_l(m;(\theta_k))|^2\nonumber\\
&=\sum_{l\in  {\cal P}_{d,q}, |l|\geq 1}h_l^{-1}|h_lQ^{\cal R}_l(u)|^{2t}h^2_l|Q_l(m;(\theta_k))|^2\nonumber\\
&\sim \sum_{l\in  {\cal P}_{d,q}, |l|\geq 1}h_l^{-1}\prod_{k=1}^{q-1}
\Big [\Big (\sum_{j=1}^{q-1}\gamma[j]\cos(2\pi/q\cdot jk)\Big )^{2t}
\Big |(1 - |\xi| )+ \sum_{j=1}^{q-1}\xi[j]\theta_1^{kj}\Big |^2\Big ]^{l_k}
\nonumber\\
&=\sum_{|l|=1}^d{d\choose |l|}
\Big [\sum_{k=1}^{q-1}
\Big (\sum_{c=1}^{q-1}\gamma[c]\cos(2\pi/q\cdot ck)\Big )^{2t}
\Big  |(1 - |\xi| )+ \sum_{a=1}^{q-1}\xi[a]\theta_1^{ka}\Big |^2
\Big ]^{|l|}\\
&= \Big [1 + \sum_{k=1}^{q-1}
\Big ( \sum_{c=1}^{q-1}\gamma[c]\cos(2\pi/q\cdot ck)\Big )^{2t}\Big |(1 - |\xi| )+ \sum_{a=1}^{q-1}\xi[a]\theta_1^{ka}\Big |^2
\Big ]^d - 1
\end{align*}
If $q$ is prime the factors to the power $2t$ are less than 1.
To obtain an upper bound
\begin{align*}
\chi^2_t(m;d) &\leq
\exp \Big \{d
\sum_{k=1}^{q-1}
\Big (\sum_{c=1}^{q-1}\gamma[c]\cos(2\pi/q\cdot ck)\Big )^{2t}
\Big |(1 - |\xi| )+ \sum_{a=1}^{q-1}\xi[a]\theta_1^{ka}\Big |^2
\Big \}-1\nonumber\\
&\leq \exp \Big \{d\sum_{k=1}^{q-1}e^{2t\big ( - 1
+ \sum_{c=1}^{q-1}\gamma[c]\cos(2\pi/q\cdot ck)\big )}
\Big \}- 1\nonumber\\
&\leq \exp \Big \{d (q-1)e^{2t\max_k \big (- 1
+ \sum_{c=1}^{q-1}\gamma[c]\cos(2\pi/q\cdot ck)\big )}\Big \}- 1\nonumber\\
&\leq \exp \Big \{ d(q-1) e^{ 2t( m^*-1)}\Big \}
- 1,
\end{align*}
where $m^* = \max_k\sum_{c=1}^{q -1}\gamma[c]\cos(2\pi/q \cdot ck)$, and we have
\[
\max_{m\in\mathcal{V}_{q,d}}\chi^2_t(m;d)\leq
\exp \Big \{ d(q-1) e^{ 2t( m^*-1)}\Big \}
- 1.
\]
Choose
\[
  t_C=\frac{1}{2(1-m^*)}\Big(\log\big(d(q-1)\big)+C\Big),
\]
where $C$ is a constant. When $t=t_C$ the upperbound is equal to
$e^{e^{-C}}-1$.
On the other hand,
\begin{align*}
\max_{m\in\mathcal{V}_{q,d}}\chi^2_t(m;d)&\ge
\chi^2_t(m_0;d)=\sum_{l\in\mathcal{P}_{d,q},|l|\ge 1}h^{-1}_l|h_lQ_l^{\mathcal{R}}(u)|^{2t}\\
&\sim \sum_{|l|=1}^d{d\choose |l|}
\Big [\sum_{k=1}^{q-1}
\Big ((1 - |\gamma| )+ \sum_{c=1}^{q-1}\gamma[c]\cos(2\pi/q\cdot ck)\Big )^{2t}
\Big ]^{|l|}\\
&\ge d(q-1)m_*^{2t}\sim d(q-1)e^{2t(m_*-1)},
\end{align*}
where $m_* = \min_k\sum_{c=1}^{q-1}\gamma[c]\cos(2\pi/q \cdot ck)$, and we have
\[
\max_{m\in\mathcal{V}_{q,d}}\chi^2_{t_C}(m;d)
\ge e^{-\frac{m^*-1}{m_*-1}C}
\big(d(q-1)\big)^{-\frac{m^*-m_*}{1-m^*}}.
\]
This random walk also exhibits cutoff as in Example~\ref{simple:000}.
\end{Example}
\begin{Remark}
  In Example~\ref{bigchi}, the number of generators is $k=2^d$ and the condition $1\ll \log k\ll \log|G|$ is not satisfied unless $q$ is large (see Remark~\ref{Cayley}). Nevertheless, we have established that the random walk exibits cutoff.
  If $1\ll \log k\ll \log|G|$ is satisfied, \cite{HOT2026} has shown that the cutoff time is determined only by $k$ and $|G|$ when $m(G)\ll \log |G|$ and $k-m(G)\asymp k\gg 1$, where $m(G)$ is the minimal size of a generating subset of $G$. On the other hand, in Example~\ref{bigchi}, the cutoff time generally depends on $\gamma[c]$, $c\in[q-1]$, and in particular for $\gamma[1]=\cdots=\gamma[q-1]=1/(q-1)$, $m^*=m_*=-1$ holds.
\end{Remark}  
\begin{Remark}\label{Hamming:23}
Let $\partial(x,y) = |\{k:x[k]\ne y[k], k \in [d]\}|$ be the Hamming distance between $x,y \in {\cal V}_{q,d}$. Identify transitions in Hamming distance by noting that in (\ref{spectral:006}) 
$\partial(x,y) = \sum_{k=1}^q n_{x-y}[k]=d - n_{x-y}[0]$. We calculate the marginal transition distribution of $n_{x-y}[0]$ from (\ref{spectral:006}) in Corollary \ref{nxycorr:01}.
Consider
\[
\sum_{n_{x-y}:\ n_{x-y}[0]\text{~fixed}}\ \frac{1}{q^{d}}\frac{d!}{n_{x-y}[0]!\cdots n_{x-y}[q-1]!}Q_l(n_{x-y};(\theta_k))
\]
which is the coefficient $w^l$ in 
\begin{align}
&\sum_{n_{x-y};\  n_{x-y}[0]\text{~fixed}}
\ \frac{1}{q^{d}}\frac{d!}{n_{x-y}[0]!\cdots n_{x-y}[q-1]!}
\prod_{j=0}^{q-1}\Big (1 + \sum_{k=1}^{q-1}w_k\theta_k^j\Big )^{n_{x-y}[j]}\nonumber\\
&= \frac{1}{q^{d}}{d\choose n_{x-y}[0]}
\big (1 + \sum_{k=1}^{q-1}w_k\big )^{n_{x-y}[0]}
\Big ( \sum_{j=1}^{q-1} \big (1 + \sum_{k=1}^{q-1}w_k\theta^j_k\big )\Big )^{d-n_{x-y}[0]}\nonumber\\
&= \frac{1}{q^{d}}{d\choose n_{x-y}[0]}
\big (1 + \sum_{k=1}^{q-1}w_k\big )^{n_{x-y}[0]}
\big ( q-1 - \sum_{k=1}^{q-1}w_k\big )^{d-n_{x-y}[0]}.
\label{calc:766}
\end{align}
Note that $\sum_{j=1}^{q-1}\theta^j_k=-1$ for $k \in [q-1]$ to obtain the third line from the second line in 
(\ref{calc:766}).
The coefficient of $w^l$ in (\ref{calc:766}) is equal to 
\begin{align*}
&{d\choose n_{x-y}[0]}\Big (\frac{1}{q}\Big )^{n_{x-y}[0]}
\Big (1-\frac{1}{q}\Big )^{d-n_{x-y}[0]}\nonumber\\
&~~\times{|l|\choose l}{d\choose |l|}Q^K_{|l|}(d - n_{x-y}[0];d,1-1/q).
\end{align*}
The transition distribution is therefore
\begin{align}
\mathbb{P}(n_{x-y}[0]) =
&{d\choose n_{x-y}[0]}\Big (\frac{1}{q}\Big )^{n_{x-y}[0]}
\Big (1-\frac{1}{q}\Big )^{d-n_{x-y}[0]}
\nonumber\\
&~~\times\Big \{1 + \sum_{|l|=1}^d\gamma_{|l|}{d\choose |l|}
Q^K_{|l|}(d - n_{x-y}[0];d,1-1/q)\Big \},
\label{hmarkov:00}
\end{align}
where 
\begin{align*}
\gamma_{|l|} = \sum_{l:|l|\text{~fixed~}\leq d}{|l|\choose l}\kappa_l
= {d\choose |l|}^{-1}
\sum_{l:|l|\text{~fixed~}\leq d}
\mathbb{E}\big [Q_l(V;(\theta_k))\big ]
\end{align*}
Denote the Hamming distance at time $t$ from all states initially zero  as
 $\partial_t(h)=d - n_{0-h}^{(t)}[0]$.
The distribution of $\partial_t(h)$, derived in a similar way to (\ref{hmarkov:00}) is
\begin{align}
\mathbb{P}(\partial_t(h)) =
&{d\choose \partial_t(h)}
\Big (1-\frac{1}{q}\Big )^{\partial_t(h)}\Big (\frac{1}{q}\Big )^{d-\partial_t(h)}
\nonumber\\
&~~\times\Big \{1 + \sum_{|l|=1}^d\gamma_{|l|}^{(t)}{d\choose |l|}
Q^K_{|l|}(\partial_t(h);d,1-1/q)\Big \},
\label{hmarkov:00a}
\end{align}
where 
\begin{align*}
\gamma_{|l|}^{(t)} &= \sum_{l:|l|\text{~fixed~}\leq d}{|l|\choose l}\kappa_l^t. \end{align*}
The Hamming distance from zero is Markovian if $q=2$, however this is not generally true  if $q>2$.
It will be Markovian if and only if for exchangeable random variables in $V$ with
$v_1,\ldots, v_{h} \ne 0$, $h\in [d]$,
\begin{equation}
\mathbb{P}\big (V[1]=v_1,\ldots, V[h]=v_{h}, V[h+1]=0,\ldots ,V[d]=0\big )
\label{mcondition:00}
\end{equation}
does not depend on $v_1, \ldots v_{h}$. This is the same as saying that conditional on $V$ having $d-h$ entries zero the remaining $h$ entries are uniform on $[q-1]^h$. 
Set
\begin{equation*}
q_h = {d\choose h}
\mathbb{P}\big (V[1] > 0,\ldots, V[h] > 0, V[h+1]=0,\ldots ,V[d]=0\big ),
\end{equation*}
Then $\sum_{h=0}^d q_h = 1$. By considering the hypergeometric number of entries $k$ where 
$V[j] > 0,\ r[j] >0$, and denoting $\partial(r,0)=|\big \{k:r[k]>0, k \in [d]\big \}|$,
\begin{align}
\rho_r
&
=  \mathbb{E}\big [\prod_{k \in [d],\ r[k] > 0}\theta_{r[k]}^{V[k]}\ \big ]\nonumber\\
&
= \sum_{h=0}^{\partial(r,0)}q_h
\sum_{k=0}^h
(-1)^k
\frac{
{h\choose k}{d-h\choose \partial(r,0)-k}
}
{
{d\choose \partial(r,0)}
}
\nonumber\\
&
= \sum_{h=0}^{\partial(r,0)}q_h
Q^K_{h}(\partial(r,0);d,1-1/q).
\label{rhocalc:01}
\end{align}
The factors $-1$ appear because $\sum_{v=1}^{q -1}\theta_j^v = -1$ when $j \ne 0$.\\
In the count values $l$ of $r$, $\partial(r,0) = |l|$, so from (\ref{rhocalc:01}),
\begin{equation*}
\kappa_l =  \sum_{h=0}^{|l|}q_h
Q^K_{h}(|l|;d,1-1/q).
\end{equation*}
When (\ref{mcondition:00}) holds
\begin{equation}
\gamma_{|l|}^{(t)} = (q-1)^{|l|}\kappa_l^t.
\label{gammat:00}
\end{equation}
The right side of (\ref{gammat:00}) only depends on $|l|$.

%
%
\begin{Example}
A simple random walk. An entry in $V$ is chosen uniformly to change, taking values $1,q-1$ with probability $\gamma/2,\gamma/2$ or equal to zero with probability 
$1-\gamma$.
The Hamming distance is Markovian with $q_1 = \gamma,\ q_0=1-\gamma$.\\[0.1cm]
Extending this model suppose that a random number of $A$ entries are chosen to change in the same way.
The Hamming distance is Markovian with
\begin{equation*}
q_h = {d\choose h}\sum_{a=h}^d {a\choose h}\gamma^h(1-\gamma)^{a-h}\mathbb{P}\big (A=a\big ).
\end{equation*}
$\gamma$ could be take as a random variable in these models in each transition.
\end{Example}
\cite{H1997} studies cutoff times in models where (\ref{mcondition:00}) holds. They use an algebraic method to derive (\ref{rhocalc:01}).
\cite{GM2025} uses these Markov models to construct a quantum walk. The condition (\ref{mcondition:00}) is equivalent to their assumption in the model. These models are much simpler than the general models because the eigenstructure is determined by the one-dimension Krawtchouk polynomials 
instead of the multivariate Krawtchouk polynomials. The algebraic treatment in \cite{H1997} is still effective for quantum walks. 
\end{Remark}
\subsection{Central limit form of $Q_l( m;(\theta_j),d)$ as $d\to \infty$}
\label{CLTQ}
Let $M^{(d)}$ have a uniform Multinomial$(d,(1/q))$ distribution. Scale by taking 
$M^{(d)}= \frac{d}{q} +\sqrt{d}\ \M^{(d)}$. Then it is well known that 
$\M^{(d)}$ converges weakly to a singular multivariate normal $\M$ with zero means and
\begin{equation}
 \text{Cov}(\M[a],\M[b]) = \frac{1}{q}\big (\delta_{ab}- \frac{1}{q}\big ).
 \label{cov:256}
\end{equation}
The moment generating function of $\M$ is
\begin{equation}
\mathbb{E}\big [e^{\M\dt\psi}\big ] 
= \exp \Big \{\frac{1}{2}\sum_{a,b=0}^{q-1}\frac{1}{q}\big (\delta_{ab}- \frac{1}{q}\big )\psi[a]\psi[b]\Big \}.
\label{mgf:256}
\end{equation}
Singularity occurs because $\sum_{a=0}^{q-1}\mathfrak{M}[a]=0$. Removing the first entry
$\mathfrak{M}_+=(\mathfrak{M}[1],\cdots, \mathfrak{M}[q-1])$ is nonsingular.
The covariance matrix of $\mathfrak{M}_+$ is $\frac{1}{q}\big (I - \frac{1}{q}J\big )$,
where $I$ and $J$ are $q-1\times q-1$ matrices and $J$ is a matrix of unit entries. The inverse covariance matrix is $q\big (I + J\big )$, which has one eigenvalue $q^2$ and $q-2$ eigenvalues equal to $q$, thus with determinant $q^q$.
The density of $\mathfrak{M}_+$ is, for $\mathfrak{m}_+ \in \mathbb{R}^{q-1}$ is therefore
\begin{equation*}
\varphi(\mathfrak{m}_+;q)=\frac{q^{q/2}}{(2\pi)^{(q-1)/2}}
\cdot \exp\Big \{-\frac{q}{2}\sum_{a,b=1}^{q-1}(\delta_{ab}+1)\mathfrak{m}[a]\mathfrak{m}[b]\Big \}.
\end{equation*}
The generating function for the multivariate Krawthouk polynomials (\ref{genfn:00}) can be written in terms of $\mathfrak{m}_+$ as 
\begin{equation*}
\Big (1 + \sum_{k=1}^{q-1}w_k\Big )^{d - |\mathfrak{m}_+|}\prod_{j=1}^{q-1}\Big (1 + \sum_{k=1}^{q-1}w_k\theta_k^j\Big )^{m[j]}.
\end{equation*}
\begin{Lemma}\label{QCLT}
Take a weak limit as $d\to \infty$ in $M^{(N)}$ uniform multinomial$(d,(1/q))$ to obtain a 
multivariate normal $\M$ with covariance function (\ref{cov:256}) by setting  
$M^{(d)}=d/q +\sqrt{d}\ \M$.
Let
\begin{equation}
Q_l(\m,(\theta_j);\infty) =
\lim_{d\to \infty}Q_l(d/q +\sqrt{d}\ \m);(\theta_j),d)d^{-|l|/2},
\label{logcalc:05a}
\end{equation}
where $l\in \mathbb{N}^{q-1}$. 
$\big (Q_l(\m,(\theta_j);\infty)), (\overline{Q_l(\m,(\theta_j);\infty)})$ is a biorthogonal system of orthogonal polynomials on $\M$ with 
\begin{equation}
\mathbb{E}\big [Q_l(\M,(\theta_j);\infty)\overline{Q_{l^\prime}(\M,(\theta_j);\infty)}\big \}
= \frac{\delta_{ll^\prime}}{l[1]!\cdots l[q-1]!}.
\label{bio:00}
\end{equation}
$Q_l(\m,(\theta_j);\infty)$ is the coefficient of $\w^l$ in the generating function
\begin{align}
G(\m,\w,(\theta_j);\infty) =
\exp \Big \{
-\frac{1}{2q}\sum_{j=0}^{q-1}\big(\sum_{k=1}^{q-1}\w_k\theta_k^j\big)^2
+\sum_{j=0}^{q-1}\m[j]\sum_{k=1}^{q-1}\w_k\theta_k^j\Big \}.
\label{logcalc:01a}
\end{align}
An explicit expression is
\begin{align}
Q_l(\m,(\theta_j);\infty) &=
\sum_{a\in \mathbb{N}^q, |a|=|l|}\ \prod_{j=0}^{q-1}\frac{1}{a[j]!}H_{a[j]}(\m[j];q)\nonumber\\
&\quad\quad\quad\times\frac{a[0]!\cdots a[q-1]!}{l[0]!\cdots l[q-1]!}
Q_{a^+}(l;(\theta_k), |l|),
\label{MKHermite:00}
\end{align}
where $a^+ = (a[1],\cdots,a[q-1])$ and $\big (H_k(x;q)\big )_{k=0}^\infty$ are Hermite-Chebycheff polynomials orthogonal on a $N(0,1/q)$ distribution. A generating function for them is
\begin{equation}
\sum_{k=0}^\infty \frac{1}{k!}z^kH_k(x;q)=
\exp \Big \{- \frac{1}{2q}z^2 + xz\Big \}.
\label{HCgen:00}
\end{equation}
\end{Lemma}
\begin{proof}
Substitute $w = \w/\sqrt{d}$ and 
$m= \frac{d}{q} +\sqrt{d}\m$ in the generating function (\ref{genfn:00}) to obtain
\begin{equation*}
G(\m,\w,(\theta_j);d) = \prod_{j=0}^{q-1}
\Big (1 + \sum_{k=1}^{q-1}\frac{\w_k}{\sqrt{d}}\theta_k^j\Big )^{\frac{d}{q} +\sqrt{d}\m[j]}.
\end{equation*}
Now
\begin{align*}
&\log G(\m,\w,(\theta_j);d) \nonumber\\
&= \sum_{j=0}^{q-1}\Big (\frac{d}{q} +\sqrt{d}\m[j]\Big )\Big (
\sum_{k=1}^{q-1}\frac{\w_k}{\sqrt{d}}\theta_k^j
-\frac{1}{2}\big (\sum_{k=1}^{q-1}\frac{\w_k}{\sqrt{d}}\theta_k^j\big )^2\Big )
+ {\cal O}(\frac{1}{\sqrt{d}})\nonumber\\
&= -\frac{1}{2q}\sum_{j=0}^{q-1}\big (\sum_{k=1}^{q-1}\w_k\theta_k^j\big )^2
+\sum_{j=0}^{q-1}\m[j]\sum_{k=1}^{q-1}\w_k\theta_k^j
+ {\cal O}(\frac{1}{\sqrt{d}}).
\end{align*}
Therefore $G(\m,\w,(\theta_j);\infty)
:=\lim_{d\to \infty}G(\m,\w,(\theta_j);d)$ is equal to (\ref{logcalc:01a}).
The biorthogonality relationship (\ref{bio:00}) follows from a calculation that
\begin{equation*}
\mathbb{E}\big [G(\m,\w,(\theta_j);\infty)\overline{G(\m,\u,(\theta_j);\infty)}\big ]
=
\exp \Big \{ \sum_{k=1}^{q-1}\w_k\u_k\Big \}.
\end{equation*}
The right side of (\ref{logcalc:01a}), from (\ref{HCgen:00}), is equal to
\begin{align}
&\prod_{j=0}^{q-1}\
\sum_{a[j]=0}^\infty \frac{1}{a[j]!}H_{a[j]}(\m[j];d)\big (\sum_{k=1}^{q-1}\w_k\theta_k^j\big )^{a[j]}
\nonumber\\
&= \sum_{a\in \mathbb{N}^q}\ \prod_{j=0}^{q-1}\frac{1}{a[j]!}H_{a[j]}(\m[j];q)
\times \prod_{j=0}^{q-1}\big (\sum_{k=1}^{q-1}\w_k\theta_k^j\big )^{a[j]}.
\label{logcalc:02a}
\end{align}
The last factor in (\ref{logcalc:02a}), for fixed $a$ is identified as a generating function for the multivariate Krawtchouk polynomials from (\ref{genfn:02}). Take $d=|a|$, $a^+ = (a[1],\ldots a[q-1])$ 
(so $d-|a^+| = a[0]$) and write the factor as
\begin{equation}
 (\sum_{k=1}^{q-1}\w_k\big )^{a[0]}\prod_{j=1}^{q-1}\big (\sum_{k=1}^{q-1}\w_k\theta_k^j\big )^{a[j]}
\label{logcalc:03a}
\end{equation}
The coefficient of $\w^l$, $l\in \mathbb{N}^{q-1}$ in (\ref{logcalc:03a}) is equal to
\begin{equation}
h_{a}{|a|\choose l}Q_{a^+}(l;(\theta_k), |a|)
= \frac{a[0]!\cdots a[q-1]!}{l[0]!\cdots l[q-1]!}Q_{a^+}(l;(\theta_k), |a|).
\label{logcalc:04a}
\end{equation}
Therefore (\ref{MKHermite:00}) holds from
from (\ref{logcalc:02a}) and (\ref{logcalc:04a}).
\end{proof} 
\color{black}
\begin{Remark}
The generating function (\ref{logcalc:01a}) simplifies to
\begin{equation*}
G(\m,\w,(\theta_j);\infty)
= \exp \Big \{-\frac{1}{2}\sum_{k=1}^{q-1}\mathfrak{w}_k\mathfrak{w}_{q-k}
+ \sum_{k=1}^{q-1}\mathfrak{w}_k\sum_{j=0}^{q-1}\mathfrak{m}[j]\theta_k^j\Big \}
\end{equation*}
and it is possible to give an alternative expression for $Q_l(\m,(\theta_j);\infty)$ from this generating function.
\end{Remark}

\label{proposition:clt}
\begin{Proposition} A transition function obtained as a limit from (\ref{spectral:0011}) in Proposition \ref{Generaltrans}
as $d\to \infty$, $d/q + \sqrt{d}n \to \n$, $d/q + \sqrt{d}m \to \m$, $V/d \to \v$ is
for $\m_+,\n_+\in \mathbb{R}^{q-1}$, $\v \in \Delta_q$,
\begin{align*}
&\varphi(\n_+;q)\Big \{1 + \sum_{l\in \mathbb{N}^q, |l|>0}
\prod_{k=1}^{q-1}\Big ( 1 - |\v| + \sum_{j=1}^{q-1}\v[j]\theta_k^j\Big )^{l_k}\nonumber\\
&~~~~~~~~~~~~~~~~~~~~~~~\times \prod_{k=1}^{q-1}l_k!Q_l(\m,(\theta_j);\infty)\overline{Q_l(\n,(\theta_j);\infty)}\Big \}.
\end{align*}
\end{Proposition}
\begin{proof}
From (\ref{asymptotic:99}) in Lemma \ref{Kplimit}
\begin{equation*}
\lim_{d \to \infty}
h_lQ_l(V;(\theta_k))
= \prod_{k=1}^{q-1}\Big ( 1 - |\v| + \sum_{j=1}^{q-1}\v[j]\theta_k^j\Big )^{l_k}.
\end{equation*}
and from (\ref{logcalc:05a}) in Lemma \ref{QCLT}
\begin{equation*}
Q_l(\m,(\theta_j);\infty) =
\lim_{d\to \infty}Q_l(d/q+\sqrt{d}\ \m);(\theta_j),d)d^{-|l|/2},
\end{equation*}
similarly for $Q_l(\n,(\theta_j);\infty)$. Note that
\[
h_l\sim |l|^{-d}\prod_{k=1}^{q-1}l_l!.
\]

\end{proof}

\subsection{de Finetti exchangeability}
In a de Finetti random walk (\ref{rweqn:000}) on ${\cal V}_{q,\infty}$ the entries of $Z_\tau$ are exchangeable. 
$Z_\tau$ is controlled by $V_\tau$ which has a random probability distribution $p$ with de Finetti measure $\mu$.
Take any $d$ entries within $Z_\tau$. For given $a_1,\ldots, a_d \in {\cal V}_q$ 
\[
\mathbb{P}\big (Z_\tau[1]= a_1,\ldots , Z_{\tau}[d] = a_d)
= \int_{\Delta_q}p[a_1]\cdots p[a_d]\mu(dp),
\]
which is invariant under any distinct $d$ entries. A point process of \iid random variables $(V_\tau)_{\tau=1}^\infty$ on ${\cal V}_{q,\infty}$ is generated. We could think of a point process $(\theta_1^{V_\tau})$ instead of $(V_\tau)$.
Let $\widehat{X}_0$ have all entries unity. Then 
\begin{equation*}
\widehat{X}_t = \otimes_{\tau=1}^t\theta_1^{V_\tau}.
\end{equation*}
The grouped eigenvalues are
\begin{equation}
\kappa_l =\int_{\Delta_q}\prod_{k=1}^{q-1}\Big (\sum_{j=0}^{q-1}\theta_1^{kj}p[j]\Big )^{l_k}\mu(dp).
\label{rhoval:00}
\end{equation}

If $q=2$ there is only one term $l_1$ in (\ref{rhoval:00}). Denoting the measure of $p[1]$ as $\mu_1$,\\
\[
\kappa_l = 
\int_{[0,1]} \big (p[0] - p[1]\big )^{l_1} \mu_1(dp[1])
 = \int_{[0,1]} \big (1  - 2p[1]\big )^{l_1}\mu_1(dp[1]).
\]
This is related to the spins we know about in \cite{G2025}.
\begin{Example}
In Example \ref{Example:2a} jumps are made one step to the left or right with probability $\gamma/2$ or no step with probability $1-\gamma$ where $\gamma$ is a random variable with a measure we call $\mu$. Then from (\ref{rhoval:00})
\begin{equation*}
\kappa_l = \int_{[0,1]}\prod_{k=1}^{q_\circ-1}\big (1-\gamma +\gamma \cos \big (2\pi/q\cdot k\big )\big )^{l_k}\mu(d\gamma).
\end{equation*}
\end{Example}
\section*{Acknowledgements} This paper was mainly written while the first author was visiting the second author at the Institute of Statistical Mathematics, Tachikawa, Tokyo in 2025. He thanks the Institute for their support and hospitality. Thanks to Andrew Barbour for asking about a random walk on $\mathbb{Z}_q^d$ after discussing a random walk on the hypercube $\mathbb{Z}_2^d$. Thanks to Persi Diaconis, Hiroshi Mizukawa and Andrea Collevecchio for comments on this research. The second author was supported in part by JSPS KAKENHI Grants 24K06876.

\end{document}